\documentclass{article}
\pdfoutput=1
\usepackage{amsmath,amsfonts,amsthm,amssymb,mathtools}
\usepackage{setspace}
\usepackage{fancyhdr}
\usepackage{lastpage}
\usepackage{extramarks}
\usepackage{chngpage}
\usepackage{soul,color}
\usepackage{graphicx,float,wrapfig}
\usepackage[small,nohug,heads=vee]{diagrams}

\usepackage{tikz}
\usepackage{tikz}

\usepackage{hyperref}
\usepackage{extarrows}
\usepackage{graphicx}
\hypersetup{colorlinks=true,
linkcolor=blue,
citecolor=red}
\diagramstyle[labelstyle=\scriptstyle]
\newcommand{\bO}{\mathcal{O}}
\newcommand{\spec}{\mathrm{Spec}}
\newcommand{\Hom}{\mathrm{Hom}}
\newcommand{\Class}{Special cycles on orthogonal Shimura varieties}
\newtheorem{thm}{Theorem}[section]
\newtheorem{lem}[thm]{Lemma}

\title{\textmd{\bf \Class }}
\date{}
\author{}

\begin{document}

\begin{spacing}{2.2}

\maketitle \thispagestyle{empty}


\centerline{\textbf{Abstract}}

In this paper, we investigate a general method to establish tame and norm relations for special cycles in Shimura varieties, using unitary cycles in odd orthogonal Shimura varieties as a guiding example.

\tableofcontents

\section{Introduction}

A classical example of special cycles is Heegner points on modular curves. We will generalize it to higher dimensional Shimura varieties. More precisely, we will consider higher dimensional orthogonal Shimura varieties. Fix a totally real number field $F$, and a quadratic CM field extension $F\longrightarrow E$. We will consider a $2n+1$ dimensional $F$-linear space $V$ equipped with a quadratic form $Q$ with signature $(2,2n-1)$ at a real place of $F$ and $(0,2n+1)$ at all other real places, such that it contains an $n$-dimensional $E$-hermitian hyperplane $W$. We define $\textbf{G}=\mathrm{Res}_{F/\mathbb{Q}}(SO(V))$, $\textbf{H}=\mathrm{Res}_{F/\mathbb{Q}}(U(W))$, in this way we get an embedding of Shimura datum  $Sh_{\textbf{H}} \longrightarrow  Sh_{\textbf{G}}$, choose a small enough level group $K$ for $\textbf{G}(\mathbb{A}_f)$ so that we get a map between Shimura varieties $Sh_\textbf{H}(H(\mathbb{A}_f)\bigcap K)\longrightarrow Sh_\textbf{G}(K)$. The image of a geometric connected component of the unitary Shimura varieties defines a special cycle on the orthogonal Shimura varieties (for suitable K, this is a closed embedding already, although we will not need this fact). Then the basic idea is to apply $\textbf{G}(\mathbb{A}_f)$ translation to get a family of special cycles. Inside this family, we will construct special cycles with suitable distribution properties.

The first problem is to understand the Hecke action and the Galois action on this family of special cycles. We will use Deligne's reciprocity law on the set  of geometric connected components to describe the Galois action. To construct special cycles with the desired tame relations and norm relations, we will first study them locally. In the local case, we need to study the \textbf{H}-action on the Bruhat-Tits buildings of \textbf{G}. Inspired by work of Boumasmoud and Loeffler, we first obtain some purely group theoretical relations ("\textbf{the abstract relations}").

The next problem is to translate these purely group theoretical relations into "true" distribution relations ("\textbf{the realization}"). For this step, we require more conditions. For the tame relation, we need  the congruence conjecture to relate the Hecke polynomial to the Galois polynomial. For the norm relation, we need  conductor growth conditions, and ordinary conditions.

Although for simplicity, we only discuss orthogonal Shimura varieties in the main part, this approach is flexible and works in other cases.  We give some other examples as an application, $Gsp_4$ Shimura varieties and unitary GGP pair Shimura varieties. The unitary case was largely solved by Boumasmoud (see \cite{reda2019}) already. Our framework can produce similar results and also works for the corresponding similitude groups.

The idea of applying $\textbf{G}(\mathbb{A}_f)$ translation is not new. Loeffler, Cornut and many other people have already used this idea to construct Euler systems (see \cite{loeffler2017}, \cite{loeffler2020}, \cite{graham2020}, \cite{cornut2018}, \cite{reda2019}, \cite{jetchev2014}). Our method is different from Loeffler's  methods. He changes the level group of $\textbf{G}(\mathbb{A}_f)$ to realize the Galois action, and for the tame relation, one of the key point is to use a multiplicity one result, which does not hold in our setting. Instead we follow the method of Cornut,  use geometric connected components, and the Galois group acts on this set through Deligne's reciprocity law. The purely group theoretical relations are established using Boumasmoud's "Seed Relation" (see \cite{reda2019}).  We then need to connect these abstract local relations to genuine relations.  For the tame relations, this is  about understanding the Galois action on $\pi_0(Sh_\textbf{H})$. For the norm relations, Loeffler has developed a general method to construct norm compatible systems (see \cite{loeffler2019}),  and our realization step for the norm relation is inspired by his paper. Along the way, we find some special properties about the stabilizer group, and we conjecture it holds in general. This stabilizer conjecture will help us calculate conductors. We put this conjecture in the appendix because  it may have independent interest in representation theory.

More important problems are about arithmetic applications and further developments. The main arithmetic application is about the Bloch-Kato conjecture and Iwasawa theory. Cornut has deduced rank one result for Selmer group in our setting under some technical conditions (see \cite{cornut2018}). His result needs  nontriviality conditions, which is widely open at present. We would need an explicit reciprocity law, or a Gross-Zagier formula, to relate the resulting special cycles to $L$-functions ($p$-adic or complex). We also wish to use this family of special cycles to study Iwasawa Main conjectures.

The structure of this paper is as follow. After this introduction, the second section gives the general setting. In its first subsection, we work globally, establish some properties of such family of special cycles and show how to translate global problems into local setting. In its second subsection, we work in a local setting, state  Boumasmoud's relation, which is the key tool for both relations. The next two sections respectively consider the tame and norm relations. They follow the same pattern: in both cases, we first establish an abstract version of the desired relation.  In the fifth section we look at  other examples and  arithmetic applications. We construct other examples, for $Gsp_4$ or unitary type Shimura varieties. Then we deal with cohomology theory to realize special cycles as special elements in Galois representations (Euler system). Then we discuss some further questions. In the appendix, we discuss the stabilizer conjecture and prove it for some important cases, including  our main example and for symmetric pairs $(H,G)$.

\textbf{Acknowledgement} Firstly I thank  my director Christophe Cornut, for his great patience and useful help. I have also benefited very much from Loeffler's and Boumasmoud's idea. And I  thank  my friends, Wu Zhixiang, Xu Kai, Zou Jiandi, for their encouragement and great help.
This research  has received funding from the European Union's Horizon 2020 research and innovation programme  under the Marie Sklodowska-Curie grant agreement No 754362. \begin{figure}[h]\includegraphics[width=0.07\textwidth]{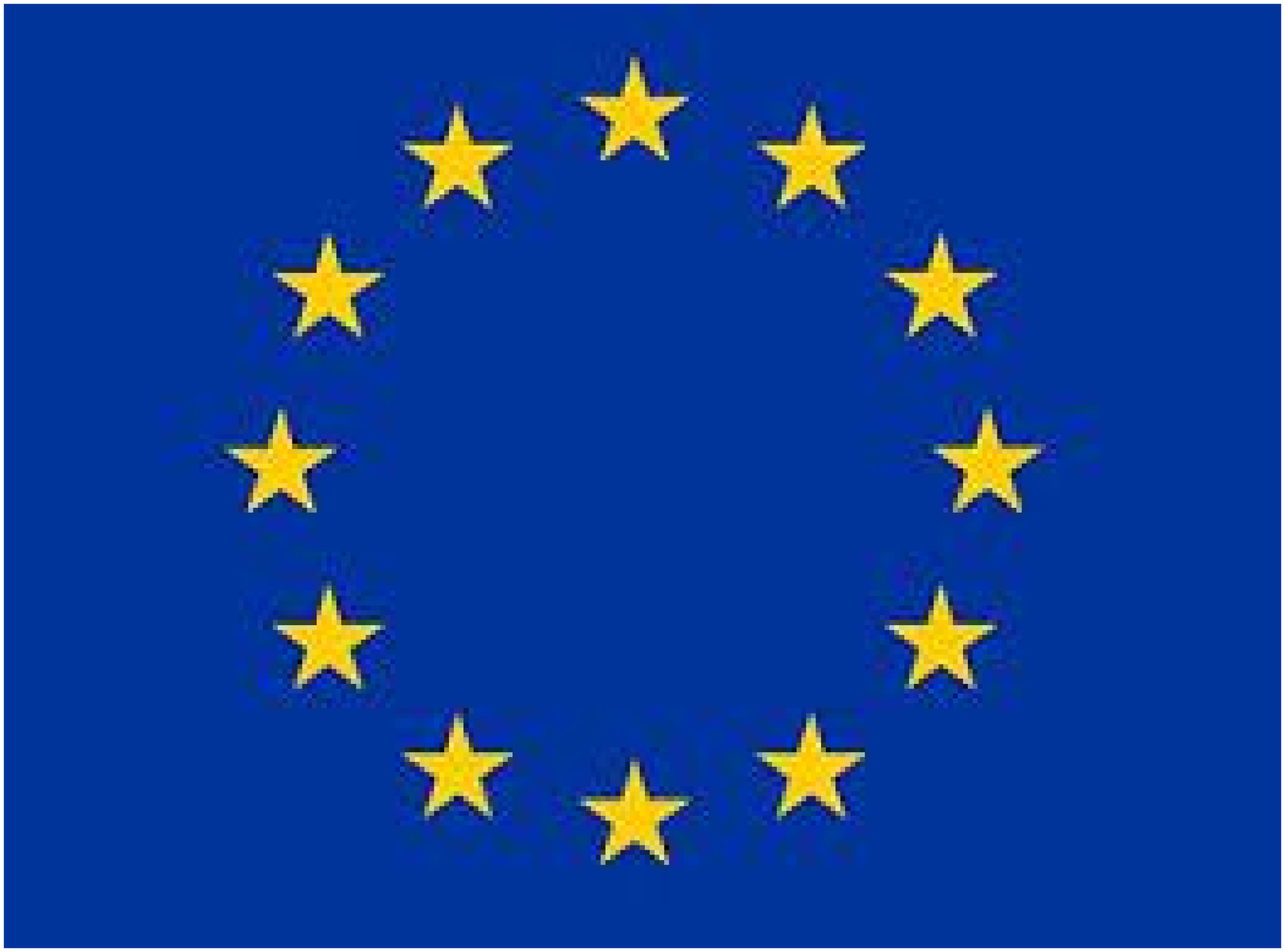}\end{figure}

\newpage

\section{General setup}

\subsection{Global setting}
\label{gloset}

Consider an embedding of Shimura datum $Sh_{\textbf{H}}\longrightarrow Sh_{\textbf{G}}$, where $\textbf{H}\longrightarrow \textbf{G}$ is an embedding of reductive groups over $\mathbb{Q}$ inducing compatible maps $\mathrm{Res}_{\mathbb{C}/\mathbb{R}}\mathbb{G}_m(\mathbb{R} )\longrightarrow \textbf{H}({\mathbb{R}}) \longrightarrow \textbf{G}({\mathbb{R}})$. Suppose the associated hermitian symmetric domain for $Sh_{\textbf{H}}$ is $Y$, the connected component is $Y^{0}$, the hermitian symmetric domain for $Sh_\textbf{G}$ is $X$. For a small enough level group $K$ for $\textbf{G}(\mathbb{A}_f)$ (here $\mathbb{A}_f$  denotes the finite adeles of $\mathbb{Q}$), we obtain Shimura varieties $Sh_\textbf{G} (K)=\textbf{G}(\mathbb{Q})\backslash X \times \textbf{G}(\mathbb{A}_f) /K$ and $Sh_\textbf{H} (\textbf{H}(\mathbb{A}_f)\bigcap K)=\textbf{H}(\mathbb{Q})\backslash Y \times \textbf{H}(\mathbb{A}_f) / (\textbf{H}(\mathbb{A}_f)\bigcap K )$. For $g \in \textbf{G}(\mathbb{A}_f)$, $Z_K(g)=[Y^0 \times gK]$ in $Sh_\textbf{G}(K)$ defines a cycle (irreducible and closed). We thus obtain a family of special cycles $Z_K(\textbf{G},\textbf{H})=\{Z_K(g)\}$, which has the following parametrization:

\begin{lem}
\label{para}
 The natural map $Stab_{\textbf{G}(\mathbb{Q})} (Y^0)  \backslash \textbf{G}(\mathbb{A}_f)/K \longrightarrow Z_K(\textbf{G},\textbf{H})$ is a  bijection.
\end{lem}

\begin{proof}

   It is obviously surjective, so we only need to establish injectivity.

   Take $g_1,g_2 \in \textbf{G}(\mathbb{A}_f)$ with $Z_K(g_1)=Z_K(g_2)$. Then, for any $y \in Y^0$, there exists $g \in \textbf{G}(\mathbb{Q})\bigcap g_1 K g_2 ^{-1}$ such that $y \in gY^0$, therefore $Y^0=\displaystyle \bigcup _{g \in \textbf{G}(\mathbb{Q})\bigcap g_1 K g_2 ^{-1}} Y^0 \bigcap gY^0$. Notice that $Y^0$ and its translation $g(Y^0)$ are  closed submanifolds of $X$, each intersection $Y^0 \bigcap g(Y^0)$ is closed in $Y^0$ and this is a countable union, so by the Baire Category theorem, there exists some $g$ in $\textbf{G}(\mathbb{Q})\bigcap g_1 K g_2 ^{-1}$ such that  $Y^0 \bigcap gY^0$ contains a non-empty open subset $U$ of $Y^0$. Take a point $y \in U$, we get equalities among tangent spaces $T_{y}Y^0=T_{y}U=T_{y}(gY^0)$ by dimension reason. Because both $Y^0$ and $gY^0$ are totally geodesic submanifolds, this implies $gY^0=Y^0$. Therefore $g \in Stab_{\textbf{G}(\mathbb{Q})}(Y^0)$.
\end{proof}

Now we can introduce the main object of this paper, the free $R$-module generated by $Z_K(\textbf{G},\textbf{H})$, $R[Z_K(\textbf{G},\textbf{H})]$, where $R$ is a suitable coefficient ring (such as $\mathbb{Z}$ or $\mathbb{Z}_p$). It will be a Galois-Hecke module. For the Hecke action, we can use the geometric Hecke action on cycles, or equivalently (see \cite{cornut2018} section 5.16), think of the global Hecke algebra as $End_{\textbf{G}(\mathbb{A}_f)}(R[G(\mathbb{A}_f)/K])$, which then has a right action on this module formally (via the parametrization). We will use this second description for the Hecke action. For the Galois action, we will use the reciprocity law for the set of geometric connected components. We have another description of $Z_K(g)$ as follow:

Consider $K_1=gKg^{-1} \bigcap \textbf{H}(\mathbb{A}_f) $, we have natural maps $Sh_\textbf{H}(K_1)\longrightarrow Sh_\textbf{G}(g K g^{-1})\xrightarrow {[*g]} Sh_\textbf{G}(K)$. These maps are defined over the reflex field E for $Sh_\textbf{H}$ and $Z_K(g)$ is the image of the connected component $[Y^0 \times K_1]$ of $Sh_\textbf{H}(K_1)_{\mathbb{\overline{Q}}}$ in $Sh_\textbf{G}(K)_{\overline{\mathbb{Q}}}$. The Galois group $Gal_E$ will act on $\pi_0(Sh_\textbf{H}(K_1))$ by Deligne's reciprocity law (see \cite{Deligne1971}), then it acts on the set of special cycles $Z_K(\textbf{G},\textbf{H})$.

Now let's come to the main example. We construct this pair briefly and refer to Cornut's paper (see \cite{cornut2018}) for more details (computation of stabilizer, reflex fields and the reciprocity law).

Take a totally real number field $F$ and fix a real place $f \in \textbf{Sp}=\spec(F)(\overline{\mathbb{Q}})=\spec(F)(\mathbb{R})=\spec(F)(\mathbb{C})$, here $\overline{\mathbb{Q}}$ is the algebraic closure of $\mathbb{Q}$ in $\mathbb{C}$. Fix a positive integer $n$ and let $(V,\phi )$ denote a quadratic $F$-vector space of dimension $2n+1$. Define $(V_a,\phi_a)=(V,\phi)\otimes _{F,a} \mathbb{R}$ ($a \in \textbf{Sp}$). We require the following signature condition:

\[sign(V_a,\phi_a)=\begin{dcases}
(2n-1,2), & \textrm{ if } a=f,\\
(2n+1,0), & \textrm{ if } a \neq f.
\end{dcases}
\]

 Define $\underline{\textbf{G}}=SO(V,\phi)$, it is a reductive group over $F$. Define $\textbf{G}=\mathrm{Res}_{F/\mathbb{Q}} \underline{\textbf{G}}$, it is a reductive group over $\mathbb{Q}$. The signature condition determines $\displaystyle  \textbf{G}_{\mathbb{R}}=\prod_a \textbf{G}_a$ with $\textbf{G}_a = SO(V_a,\phi_a)$, so that

 \[SO(V_a,\phi_a)=\begin{dcases}
 SO(2n-1,2) (\textrm{non-compact}), & \textrm{ if } a=f,\\
 SO(2n+1,0)  (\text{compact}), & \textrm{ if } a \neq f.
\end{dcases}
\]

Now we define the associated Hermitian symmetric domain. Let $X$ denote the space of oriented negative $\mathbb{R}$-planes in $(V_f,\phi_f)$. The action of $\textbf{G}_{f,\mathbb{R}}(\mathbb{R})=SO(2n-1,2)$ on $(V_f, \phi_f)$ induces a transitive action of $\textbf{G}(\mathbb{R})$ on $X$, and the stabilizer of any point is a maximal compact and connected subgroup. We can view $X$ as  a $\textbf{G}(\mathbb{R})$-conjugacy class of maps $h:\mathrm{Res}_{\mathbb{C}/\mathbb{R}}(\mathbb{R}) (\mathbb{G}_m)\longrightarrow \textbf{G}(\mathbb{R})$. In this way we get a Shimura variety $Sh_\textbf{G}(K)= \textbf{G}(Q)\backslash  X \times \textbf{G}(\mathbb{A}_f) / K$ (for $K$ small enough) (see \cite{cornut2018}).

Now we define the unitary subgroup. Let $E$ denote a totally imaginary quadratic extension of $F$ which splits $(V,\phi)$. Then $(V,\phi)$ contains  $E$-hermitian $F$-hyperplanes (see \cite{cornut2018} 9.2). We choose one such $E$-hermitian space $(W,\psi)$ and define $\underline{\textbf{H}}= U(W, \psi)$ (a subgroup of $\underline{\textbf{G}}$), $\textbf{H}=\mathrm{Res}_{F/\mathbb{Q}}\underline{\textbf{H}}$ (a subgroup of $\textbf{G}$). We can compute its signature over $\mathbb{R}$. Define $(W_a,\psi_a)=(W,\psi)\otimes _{F,a} \mathbb{R}$ ($a \in \textbf{Sp}$). Then $\displaystyle \textbf{H}_\mathbb{R}= \prod _a \textbf{H}_a$ with $\textbf{H}_a=U(W_a,\psi_a)$, these factors are as follow:

\[U(W_a,\psi_a)=\begin{dcases}
U(n-1,1), & \textrm{ if } a=f,\\
U(2n+1), & \textrm{ if } a \neq f.
\end{dcases}
\]

The complex embedding $f:F\longrightarrow \mathbb{C}$  has two extensions to $E$, and we fix one choice $f^{+}: E\longrightarrow \mathbb{C}$ and denote by  $f^-$  its conjugation. Identify $E_{f}=E \otimes _{F,f} \mathbb{R}$ with $\mathbb{C}$. Let $Y$ denote the space of negative $\mathbb{C}$-lines (via this identification) in $(W_f, \psi_f)$. Similarly, we obtain a Shimura varieties $Sh_\textbf{H}(K_1)=\textbf{H}(\mathbb{Q}) \backslash Y \times \textbf{H}(\mathbb{A}_f) / K_1$. The natural inclusion $\textbf{H}\longrightarrow \textbf{G}$  will induce a map between their Shimura varieties $Sh_\textbf{H}(K_1)\longrightarrow Sh_\textbf{G}(K)$ (if $K_1 \subset K$).

 Compared with the general setting above, this pair has some advantages. For example, $Y^0=Y$ ($Y$ is already connected), and the stabilizer in the parametrization of Lemma $2.1$ is exactly $\textbf{H}(\mathbb{Q})$.  Notice that the reflex field for $Sh_\textbf{G}$ is $f(F)$ and the reflex field for $Sh_\textbf{H}$ is $f^+ (E)$, so we identify the abstract number field $F$ and $E$ with their embedded reflex fields $f(F)$ and $f^+ (E)$. We refer to \cite{cornut2018} section 5 for proof of these facts. Below we briefly mention their conjugacy classes of Shimura cocharacters for later use.

Consider the torus $\textbf{T}=\mathrm{Res}_{E/\mathbb{Q}}\mathbb{G}_m=\mathrm{Res}_{F/\mathbb{Q}}\underline{\textbf{T}}$, where $\underline{\textbf{T}}=\mathrm{Res}_{E/F}\mathbb{G}_m$. It has a sub-torus $\textbf{T}^1=\ker(\textbf{T} \xrightarrow{Norm} \mathrm{Res}_{F/\mathbb{Q}}\mathbb{G}_m)=\mathrm{Res}_{F/\mathbb{Q}}\underline{\textbf{T}^1}$, where $\underline{\textbf{T}^1}=\ker(\underline{\textbf{T}} \xrightarrow{Norm} \mathbb{G}_{m,F})$. Through the natural determinant map,  $\det:\textbf{H}\longrightarrow \textbf{T}^1$, we can identify $\textbf{T}^1$ with the maximal abelian quotient $\textbf{H}^{ab}$. And $\det$ will induce a map between Shimura data $Sh(\textbf{H})\longrightarrow Sh(\textbf{T}^1)$. This quotient map $\det$ has non-canonical sections $\textbf{T}^1 \hookrightarrow \textbf{H}$, thus for any $\mathbb{Q}$-algebra $R$, the map $\det:\textbf{H}(R)\longrightarrow \textbf{T}^1(R)$ is surjective.

Let $\beta=(v_1,...,v_n)$ denote an orthogonal $E$-basis for $(W,\psi)$. Then

\[\displaystyle T(\beta)\overset{def}{=}\mathrm{Res}_{F/\mathbb{Q}}(U(Ev_1)\times...\times U(Ev_n))\subset \textbf{H} \subset \textbf{G} \]

 is a maximal $\mathbb{Q}$-subtorus of both $\textbf{H}$ and $\textbf{G}$, and $T(\beta)=\mathrm{Res}_{F/\mathbb{Q}}U(Ev_1)\times...\times \mathrm{Res}_{F/\mathbb{Q}}U(Ev_n) \cong (\textbf{T}^1)^n$.

 We can describe its cocharacter group over $\mathbb{C}$ as follow:

 \[X_*(T(\beta)_{\mathbb{C}}) \cong \{f: \spec(E)(\mathbb{C})\longrightarrow \mathbb{Z}^n|f(c)+f(\overline{c})=0\},\]  where $\overline{c}$ is the conjugation of $c$.

 Define a cocharacter $\mu$ for $T(\beta)_{\mathbb{C}}$, its corresponding function $f_{\mu}$ is as follow:

 \[c \mapsto \begin{dcases}
 (+1,0,...,0) & \textrm{ if } c=f^+,\\
 (-1,0,...,0) & \textrm{ if } c=f^{-}, \\
 (0,...,0)   & \textrm{otherwise.}
 \end{dcases}
 \]

 Then through the inclusion $T(\beta)\subset \textbf{H}$, we obtain a conjugacy class of cocharacter for $\textbf{H}_\mathbb{C}$, $[\mu_\textbf{H}]$. The conjugacy class of Shimura cocharacter for $Sh(\textbf{H})$  is exactly $[\mu_\textbf{H}]$. Similarly, we get the conjugacy class of Shimura cocharacter for $Sh(\textbf{G})$, $[\mu_\textbf{G}]$.

Now we explicit  the reciprocity law for $\pi_0(Sh_\textbf{H})$. It is simpler than the general case because the derived subgroup $\textbf{H}^{der}$ is simply connected,  see \cite{milne} $5.17$. By the strong approximation theorem, $\textbf{H}^{der}(\mathbb{Q})$ is dense in $\textbf{H}^{der}(\mathbb{A}_f)$. Because $E$ is a CM field,  $\textbf{T}^1(\mathbb{Q})$ is discrete (thus closed) in $\textbf{T}^1(\mathbb{A}_f)$. Combining this, we obtain $\textbf{H}(\mathbb{Q})\textbf{H}^{der}(\mathbb{A}_f)\subset \overline{\textbf{H}(\mathbb{Q})} \subset (\det^{-1})(\textbf{T}^1(\mathbb{Q}))=\textbf{H}(\mathbb{Q}) \textbf{H}^{der}(\mathbb{A}_f)$ inside $\textbf{H}(\mathbb{A}_f)$,   so $\overline{\textbf{H}(\mathbb{Q})}=\textbf{H}(\mathbb{Q})\textbf{H}^{der}(\mathbb{A}_f)$. For a small enough level group $K_1\subset \textbf{H}(\mathbb{A}_f)$, the quotient map $\det:\textbf{H}\longrightarrow \textbf{T}^1$ induces a  natural map for Shimura varieties, $Sh_{\textbf{H}}(K_1)\longrightarrow Sh_{\textbf{T}^1}(\det(K_1))$. This natural map induces an isomorphism on the set of connected components, $\pi_0(Sh_{\textbf{H}}(K_1))=\textbf{H}(\mathbb{Q}) \backslash \textbf{H}(\mathbb{A}_f)/K_1= \textbf{T}^1(\mathbb{Q})\backslash \textbf{T}^1(\mathbb{A}_f)/\det(K_1) =\pi_0(Sh_{\textbf{T}^1}(\det(K_1)))$. For a small enough level group $K\subset \textbf{G}(\mathbb{A}_f)$, similarly we have $Z_K(\textbf{G},\textbf{H})=\textbf{H}(\mathbb{Q})\textbf{H}^{der}(\mathbb{A}_f)\backslash \textbf{G}(\mathbb{A}_f)/K$. The natural map $\textbf{H}^{der}(\mathbb{A}_f)\backslash \textbf{G}(\mathbb{A}_f)/K \longrightarrow \textbf{H}(\mathbb{Q})\textbf{H}^{der}(\mathbb{A}_f)\backslash \textbf{G}(\mathbb{A}_f)/K$ induces a natural map \[{\bigotimes_{w}}^{'} \mathbb{Z}[\underline{\textbf{H}}^{der}(F_w)\backslash \underline{\textbf{G}}(F_w)/K_w] \longrightarrow \mathbb{Z}[Z_K(\textbf{G},\textbf{H})],\]  here $w$ run over primes for $F$, $K_w$ is corresponding component of $K$ ($\displaystyle K_{p}=\prod_{\mathfrak{p}|p}K_{\mathfrak{p}}$) and we use restricted tensor product. Thus we can think of $\mathbb{Z}[\underline{\textbf{H}}^{der}(F_v)\backslash \underline{\textbf{G}}(F_v)/K_v]$ as a local version of our module for special cycles.

 Recall that the reflex norm for the Shimura torus $Sh(\textbf{T}^1)$  is given by $r=res(\underline{r})$, where $\underline{r}$ $:\underline{\textbf{T}}\longrightarrow \underline{\textbf{T}^1}$, $z\mapsto \frac{z}{\overline{z}}$ (see \cite{cornut2018} section 5.8). The standard Artin map $\text{Art}_E: \textbf{T}(\mathbb{A}_{f})\longrightarrow Gal_E ^{ab}$, which sends a local uniformizer to a \textbf{geometric Frobenius}, induces an isomorphism $\text{Art}_E ^1: \textbf{T}^1(\mathbb{A}_f)/\textbf{T}^1(\mathbb{Q})\cong Gal (E[\infty]/E)$. Here $E[\infty]$ is the union of all ring class fields of $E$. The Galois group $Gal_E$ acts on $Z_K(\textbf{G},\textbf{H})$ as follow:

 \[ \rho(Z_K(g))=Z_K(sg),\, \text{if}\, \text{Art}_E^1(\det(s))=\rho|_{E[\infty]}.\]

We also need some preparations of ring class fields to realize abstract relations. For detail explanations we refer to  \cite{cornut2018} section 7 and \cite{reda2019} section VI, VII.

Fix the level group $K$, choose a special cycle $z=Z_K(g)$ and  take a finite set $S_1$ ("bad primes") of primes of $\mathbb{Q}$ large enough such that

$\bullet$ $S_1$ contains $\{2\}$;

$\bullet$ For any prime $q \notin S_1$, $\mathbb{Q}\longrightarrow E$, $\textbf{G}$ and $\textbf{H}$ are all unramified at $q$.

$\bullet$ $K=K_{S_1}\times K^{S_1}$ with $K^{S_1}=\displaystyle \prod_{q \notin S_1} K_q$, $g_q \in K_q$.

$\bullet$ For any $q \notin S_1$, $K_q$ (resp $K_q \bigcap \textbf{H}(\mathbb{Q}_q)$) is hyperspecial in $\textbf{G}(\mathbb{Q}_q)$ (resp $\textbf{H}(\mathbb{Q}_q)$).

Take a  prime $l \notin S_1 $ (we will consider $l$-adic etale cohomology in \ref{coh}) and enlarge $S_1$ into $S_1 \bigcup\{l\}$. Let $S$ ("bad places") denote the finite set of places of $F$ above $S_1$. Let $\underline{\textbf{P}}$ denote the set of primes of  $\bO_F$  that don't belong to $S$ and split in $F\longrightarrow E$,   $\underline{\textbf{N}}$ denote the set of  square-free products of elements of $\underline{\textbf{P}}$. For any ideal $m$ of $\bO_F$ belonging to $\underline{\textbf{N}}$, we will associate  a finite abelian extension $\underline{E}[m]$ that lies in $E \longrightarrow E[\infty]$.

 Choose a compact open subgroup $U_S^1\subset \textbf{T}^1(\mathbb{A}_{f,S_1})$, where $\displaystyle \mathbb{A}_{f,S_1}=\prod_{t \in S_1}\mathbb{Q}_t$. For other primes of $F$, $v \notin S$, we define a filtration (by local conductor),  $U_v^1(c)=\{\frac{a}{\overline{a}}, a \in (\bO_{F,v}+m_{F_v}^c \bO_{E,v})^*\}$, where c is a non-negative integer. Because $\textbf{T}^1(\mathbb{Q})$ is discrete in $\textbf{T}^1(\mathbb{A}_f)$, we can require $U_S^1 $ to be small enough such that $\displaystyle \textbf{T}^1(\mathbb{Q}) \bigcap (U_S^1 \times \prod_{v \notin S} U_v ^1(0))=1$. Moreover we can assume $\displaystyle U_S ^1 \times \prod_{v \notin S} U_v ^1(0) \subset \det(gKg^{-1}\bigcap \textbf{H}(\mathbb{A}_f))$ due to our assumptions on $S$.

 Now for an ideal $m$ as above, we define $U^1(m)=U_S ^1 \times \prod_{v \notin S}U_v ^1(v(m))$. We define $\underline{E}[m]$ to be the field fixed by $Art_E^1(U^1(m))$. Take $m \in \underline{\textbf{N}}$, with prime factor $h$, through direct computation, we have the following lemma:

\begin{lem}{\textbf{local-global}}
\label{locglo}

$\frac{U_h^1(0)}{U_h^1(1)}$  is isomorphic to $Gal(\underline{E}[m]/\underline{E}[\frac{m}{h}])$ via the Artin map.

\end{lem}

Although this lemma is simple, it is essential to translate the global distribution relation problem into local problems. By our assumption on $U^1(1)$, the cycle $z=Z_K(g)$ lies in the field $\underline{E}[1]$. In  section \ref{tamereal}, we will extend it into a family of special cycles defined over some $\underline{E}[m]$ with suitable tame relation.

\subsection{Local setting (Boumasmoud's Relation)}
\label{locset}

Let $F$ denote a $p$-adic field with ring of integers $\bO_F$, a uniformizer $\pi  \in  \bO_F$ and residue field $\bO_F / \pi = \mathbb{F}_q$ with $q$ elements.

 Let $\textbf{G}$ denote  a reductive group scheme over $\bO_F$. By \cite{sga3} (XXVI 7.15) and Lang's theorem \cite{lang}, $\textbf{G}$  is quasi-split over $\bO_F$: there is a maximal torus $\textbf{T}$ of $\textbf{G}$ contained in a Borel subgroup $\textbf{B}$ with unipotent radical $\textbf{N}$ and Levi decomposition $\textbf{B}=\textbf{T} \ltimes \textbf{N}$. Let $\textbf{S}$ denote the maximal split subtorus of $\textbf{T}$, $X_*(\textbf{S})$ the group of cocharacters of $\textbf{S}$ and $X^+_*(\textbf{S})\subset X_*(\textbf{S})$ the cone of $\textbf{B}$-dominant cocharacters. Let $K$ denote the hyperspecial subgroup $\textbf{G}(\bO_F)$, $B$ denote $\textbf{B}(F)$, $G$ denote $\textbf{G}(F)$; we have the Iwasawa decomposition $G=BK$.

Consider the   $G$ module $\mathbb{Z}[G/K]$, where $G$ acts on $G/K$ via left multiplication. We get the Hecke algebra $He \overset{def}{=} (End_{G} \mathbb{Z}[G/K])^{opp}= \mathbb{Z}[ K \backslash G/K] $. It has right actions on $\mathbb{Z}[G/K]$. More precisely, for a double coset $KgK=\coprod_{i} g_iK$, this operator $[KgK]$ will send $1_{bK}$ into $\sum_i 1_{bg_iK}$.

For a cocharacter $\mu \in X_* ^{+}(\textbf{S})$, there is an associated $U$-operator $U_{\mu} \in End_{B}(\mathbb{Z}[G/K])$. Let $I \subset K$ denote the Iwahori subgroup determined by $K$ and $B$, with positive part $I^+=\textbf{N}(\bO_F)$.

  By the Iwasawa decomposition, $\mathbb{Z}[G/K]$  is generated by elements of the form $1_{bK}$ ($b \in B$), $U_{\mu}$ acts on these elements as follow:

\begin{equation}
\displaystyle U_{\mu}(1_{bK})=\sum_{z \in \frac{I^+}{\mu(\pi)I^+ \mu(\pi)^{-1}}} 1_{bz\mu(\pi)K}
\label{reda1}
\end{equation}

More generally, fix an algebraic closure $\overline{F}$ of $F$ and let $F^{un}$ denote the maximal unramified extension of $F$ inside $\overline{F}$. For a conjugate class of cocharacters $[\mu]$ for $\textbf{G}_{\overline{F}}$, we  also have an associated $U$-operator. Let $\mu \in X_*(\textbf{T}_{\overline{F}})$ denote the unique $\textbf{B}_{\overline{F}}$-dominant cocharacter of $\textbf{T}_{\overline{F}}$ in this conjugacy class. Both $[\mu]$ and $\mu$ have the same field of definition, a finite extension $F(\mu)\subset \overline{F}$ of $F$. In fact, $\textbf{G}_{F}$ is an unramified reductive group over $F$ by \cite{sga3} XXVI 7.15, it splits over $F^{un}$, thus $F(\mu)\subset F^{un}$. Let $n(\mu)=[F(\mu):F]$ be the degree of this extension. We still denote the descended cocharacter  $\mathbb{G}_{m,F(\mu)}\longrightarrow \textbf{T}_{F(\mu)}$ by $\mu$. Consider its norm, $\mu_0=Norm_{F(\mu)/F} (\mu)$, it is a cocharacter for $\textbf{T}_{F}$. Therefore it factors through the maximal split torus $\textbf{S}_F$, we still denote this cocharacter $\mathbb{G}_{m,F}\longrightarrow \textbf{S}_{F}$ by $\mu_0$. There is a unique cocharacter for $\textbf{S}$ extending $\mu_0$. And it lies in $X^+ _*(\textbf{S})$. Therefore we can attach to this cocharacter a $U$-operator as above. We still denote the resulting operator by $U_{\mu}$.

Here we make three remarks:

$\bullet$ In this section we start with a reductive group scheme $\textbf{G}$ over the integer ring $\bO_F$.  This amounts to give  an unramified reductive group $G$ over $F$ and a hyperspecial subgroup $K$ of $G(F)$.

$ \bullet $ For simplicity, we pin down everything ($K,B,\mu...$). In fact, this $U$-operator has a more intrinsic and geometric explanation via Bruhat-Tits building theory. We refer to \cite{reda2019} (section III and section V) for more details.

$\bullet $ For any non-negative integer $i$, we have $U_{\mu^{i}}=(U_{\mu})^{i}$.

For the above conjugacy class of cocharacters $[\mu]$, we can also define a Hecke polynomial $Hep_{\mu}(X) \in He[X]$. There is an issue about the coefficient ring. To realize the Satake transform, we need to enlarge the coefficient ring $\mathbb{Z}$ into a larger ring $R$ such that $\mathbb{Z}[q^{\pm \frac{1}{2}}]\subset R$. We still denote the $R$-coefficient Hecke algebra $R[K \backslash G / K]$ by $He$. By scalar extension, we get a $R[G]$-module $R[G/K]$, and an associated  $U$-operator $U_{\mu}\in End_{B}R[G/K]$.

Let $\Gamma$ denote $Gal(F^{un}/F)$ with $\sigma \in \Gamma$ being the \textbf{geometric Frobenius} of $F$. Let $\rho\in X^*(\textbf{T}_{\overline{F}})\otimes \mathbb{Q}$ be the half-sum of all positive roots of $(\textbf{T}_{\overline{F}},\textbf{B}_{\overline{F}})$.

Let $1\longrightarrow \widehat{G}\longrightarrow {^L G} \longrightarrow \Gamma \longrightarrow 1$ be the Langlands dual of $\textbf{G}_{F}$. Fix a $\Gamma$-invariant pinning $(\widehat{T},\widehat{B},...)$ of $\widehat{G}$ so that $^L G=\widehat{G} \rtimes \Gamma$. We get a $\Gamma$-equivariant isomorphism $X_{*}(\textbf{T}_{\overline{F}}) \cong X^*(\widehat{T})$. It maps $\mu$ to a $\widehat{B}$-dominant character of $\widehat{T}$ fixed by $\Gamma^{n(\mu)}$, which we also denote by $\mu$. Let $r_{\mu}:\widehat{G} \rtimes \Gamma^{n(\mu)}\longrightarrow GL(V_{\mu})$ be the unique irreducible representation of $^L (\textbf{G}_{F(\mu)})$ whose restriction to $\widehat{G}$ is the irreducible representation with highest weight $\mu$, and such that $\Gamma^{n(\mu)}$ acts trivially on the highest weight space. For $\widehat{g}\in \widehat{G}$, consider the polynomial

\begin{displaymath}
\det(X-q^{n(\mu)d(\mu)}  r_{\mu}((\widehat{g} \rtimes \sigma^{-1})^{n(\mu)})),
\end{displaymath}
where $ d(\mu)= \langle \rho,\mu \rangle$. We denote it by $Hep_{\mu}$ and it is our Hecke polynomial. Its coefficients are regular functions on $\widehat{G}$ fixed by $\sigma$-conjugation. Through Satake transform, we can view these functions as elements in the Hecke algebra (with suitable coefficient ring). The Hecke polynomial only depends on the conjugacy class $[\mu]$, not depends on our pinning down data $(\textbf{B},\textbf{T},\mu...)$.

\textbf{Remark:}

Here we will use the usual ("untwisted") Satake transform and we make a remark about the coefficient issue. If $\mu$ is minuscule, Wedhorn  showed that the coefficients of its Hecke polynomial lie in the Hecke algebra with $\mathbb{Z}[q^{-1}]$-coefficient, see  \cite{wed} section 2.8. Because the Shimura cocharacter is minuscule, in the application to tame relations, we can relax the requirement $\mathbb{Z}[q^{\pm \frac{1}{2}}]\subset R$ into $\mathbb{Z}[q^{-1}]\subset R$.

Now we can state Boumasmoud's result.

\begin{thm}{(Seed Relation)}
\label{reda2}

The operator $U_{\mu}$ is a right root of the Hecke polynomial $Hep_{\mu}(X)$ in $End_{B}(R[G/K])$, in other words, in that ring, $Hep_{\mu}(U_{\mu})=0$.

\end{thm}

We refer to Boumasmoud's thesis (see \cite{reda2019} section IV) for its proof.

There are many occurences of the Hecke polynomial in the literature. Different papers may use different notations. To avoid such notation confusion, we conclude this section with some \textbf{remarks}:

$\bullet$ In this paper, the word "Frobenius" will always means geometric Frobenius, while some papers are using arithmetic Frobenius to define Hecke polynomial, like \cite{reda2019}, \cite{cornut2018}.

$\bullet$ We are using the usual Satake transform. In Boumasmoud's thesis, he used the usual Satake transform and a "twisted" version, but he named the latter "untwisted Satake transform", see \cite{reda2019} section III and section IV.

$\bullet$ In the definition of the Hecke polynomial, we consider the representation $V_{\mu}$, but some papers are using $V_{-\mu}$, see \cite{lee2020} remark 2.1.3.

$\bullet$ For such module $R[G/K]$, we always equip it with \textbf{left} $G$-action and \textbf{right} $He$-action, like \cite{reda2019}. Some papers may define the Hecke algebra as $End_{G}(R[G/K])$, while we're using the opposite identification $He=(End_{G}(R[G/K]))^{opp}$. Especially there are some notation confusion when compare our "formal" Hecke action with geometric Hecke action. See \cite{lee2020} remark 6.1.7.

\section{Tame relation}

Cornut has already established the tame relation for inert places (see \cite{cornut2018} section 7,8), which is sufficient for the main arithmetic applications, i.e. the construction of an Euler system. However, here we want to establish tame relations for split places via the seed relation, because this method can be applied to other Shimura varieties (such as $Gsp_4$ lift and so on).

\subsection{Abstract relation}
\label{abstamerel}

In this section, we still work over a local field and use the same notations as in the previous section \ref{locset}.

Recall  $F$ is a $p$-adic field with ring of integers $\bO_F$, a uniformizer $\pi  \in  \bO_F$ and residue field $\bO_F / \pi =\mathbb{F}_q$ with $q$ elements.

Recall $\textbf{G}$ is the reductive group scheme over $\bO_F$. For simplicity in this section we require it to be \textbf{split}. Then the maximal split subtorus  $\textbf{S}$ is the maximal torus $\textbf{T}$. Notice that the natural map $X_*(\textbf{T})\longrightarrow X_*(\textbf{T}_{F})$ is an isomorphism. Consider a cocharacter $\mu \in X_*(\textbf{T}_F)$, we also use $\mu$ to denote its unique extension to integral level. We require $\mu$ to be $\textbf{B}_F$-dominant ($\mu \in X^+_*(\textbf{T})$) and minuscule.

The quotient map $\bO_F^*\longrightarrow \mathbb{F}_q^*$ has a natural section $\mathbb{F}_q^*\longrightarrow \bO_F^*$. Combining it with the cocharacter $\mu$, we obtain a group map $\mathbb{F}_q^*\longrightarrow \bO_F^*\longrightarrow \textbf{T}(\bO_F) \longrightarrow \textbf{G}(\bO_F)$. Because $\mu$ is minuscule, this map is injective. We identify $\mathbb{F}_q^*$ with $\mu(\mathbb{F}_q^*)$ inside $\textbf{G}(\bO_F)$. And we let $\mu(\mathbb{F}_q^*)$ act on  $R[G/K]$  through our left $G$ action and denote elements $1_{bK}$ by $[b]$.

We have the following divisibility lemma:

\begin{lem}

$\displaystyle q-1 | Hep_{\mu}(\mu(\pi))([1])$ in  $R[\mu(\mathbb{F}_q^*)\backslash G /K]$ .

\end{lem}

\begin{proof}

There is a natural $\textbf{T}(F)$-Hecke equivariant map $pr: R[G/K] \longrightarrow R[\mu(\mathbb{F}_q^*)\backslash G/K]$. Recall that we have the seed relation (theorem \ref{reda2}), $Hep_{\mu}(U_{\mu})=0$ in $End_{B}R[G/K]$. To prove our lemma, we need to show:
\[ q-1|pr(Hep_{\mu}(U_{\mu})([1])-Hep_{\mu}(\mu(\pi))([1])).\]
Suppose the Hecke polynomial $Hep_{\mu}$ equals $\sum_i A_i X^{i}$ ($A_i \in He$) and recall that  $U_{\mu}^{i}=U_{\mu^{i}}$. It is enough to show the following:

For each positive integer $i$, we have \[q-1|pr((U_{\mu^{i}}-\mu^{i}(\pi))([1])).\]
Recall our formula for $U_{\mu^{i}}$ operator (formula (\ref{reda1})), we have:
\[U_{\mu^{i}}([1])=\sum_{z \in \frac{I^{+}}{\mu(\pi)^{i}I^+ \mu(\pi)^{-i}}}[z \mu(\pi)^{i}].\]
Denote this indexing set by $In$ and we need to analyze it. Recall $I^+=\textbf{N}(\bO_F)$. Over $\spec(\bO_F)$, for each positive root $\alpha$, there is an associated root subgroup $\textbf{N}_{\alpha}\subset \textbf{N}$ and a $\textbf{T}$-equivariant isomorphism $R_{\alpha}: \mathbb{G}_a \longrightarrow \textbf{N}_{\alpha} $, here $\textbf{T}$ acts on $\mathbb{G}_a$ through the cocharacter $\alpha:\textbf{T} \longrightarrow \mathbb{G}_m$ and $\textbf{T}$ acts on $\textbf{N}_{\alpha}$ through conjugation. Let $\Phi^{+}$ denote the set of positive roots and consider the  $\textbf{T}$-equivariant product map:
\[ \mathbb{G}_a^{n}= \prod_{\alpha \in \Phi^{+}}\mathbb{G}_a \cong \prod_{\alpha \in \Phi^{+}} \textbf{N}_{\alpha}\longrightarrow \textbf{N}.\]
This map is not compatible with the group structures, but it is an isomorphism of schemes for any ordering on the factors.

Divide these factors into two  part, define $\displaystyle \textbf{A}=\prod_{(\alpha,\mu)=1}\textbf{N}_{\alpha}$, $\displaystyle \textbf{B}=\prod_{(\beta,\mu)=0}\textbf{N}_{\beta}$. The above isomorphism induces an isomorphism $\textbf{A}(\bO_F) \times \textbf{B}(\bO_F) \cong \textbf{N}(\bO_F)$.

Notice that for positive root $\alpha$ and $\gamma$ with $(\alpha,\mu)=1=(\gamma,\mu)$, there is no positive root $\xi$ in the form of $a\alpha+b\gamma$ ($a$ and $b$ are positive integers). Therefore, $\textbf{A}$ is a commutative subgroup of $\textbf{N}$. For any positive integer $i$, we have a natural isomorphism $\textbf{A}(\bO_F/\pi^{i})\cong \frac{\textbf{A}(\bO_F)}{\mu(\pi)^{i}\textbf{A}(\bO_F)\mu(\pi)^{-i}}$. The inclusion $\textbf{A}(\bO_F)\hookrightarrow \textbf{N}(\bO_F)$ also induces a natural isomorphism $\frac{\textbf{A}(\bO_F)}{\mu(\pi)^{i}\textbf{A}(\bO_F)\mu(\pi)^{-i}} \cong \frac{\textbf{N}(\bO_F)}{\mu(\pi)^{i}\textbf{N}(\bO_F)\mu(\pi)^{-i}}$.

In summary, for the indexing set, we have found a $\textbf{T}(\bO_F)$-equivariant isomorphism

\[\prod_{(\alpha,\mu)=1}\textbf{N}_{\alpha}(\bO_F/\pi^{i}) \cong In.\]

Define $In^*$ to be $In-\{1\}$. Then $\displaystyle U_{\mu^{i}}([1])-\mu(\pi)^{i}([1])=\sum_{z\in In^*}[z\mu(\pi)^{i}]$.

For any $x \in In^*$, write $\displaystyle x=\prod_{(\alpha,\mu)=0} \textbf{R}_{\alpha}(x_{\alpha})$. For any $t\in \mathbb{F}_q^*$, $\displaystyle \mu(t)x\mu(t)^{-1}=\prod_{(\alpha,\mu)=0} \textbf{R}_{\alpha}(tx_{\alpha})$. There is at least one $\alpha$ such that $x_{\alpha}$ is non-zero, so the group  $\mu(\mathbb{F}_q^*)$ acts on $In^*$ freely. Notice that in $\mu(\mathbb{F}_q^*) \backslash G/K$ we have $pr([x\mu(\pi)^{i}])=pr([\mu(t)x\mu(\pi)^{i}\mu(t)^{-1}])=pr([\mu(t)x\mu(t)^{-1}\mu(\pi)^{i}])$, here we use $\mu(\mathbb{F}_q^*)\subset K$.  Because the cardinality of $\mu(\mathbb{F}_q^*)$ is $q-1$, we get $q-1|pr((U_{\mu^{i}}-\mu^{i}(\pi))([1]))$.
\end{proof}
We make a remark. In fact, we can replace the element $[1]$ by any element $[t]$, where $t$ is an element in $\textbf{T}(F)$, this lemma still holds with the same proof.

Now we will translate this lemma into the relative setting.

Suppose there exists a closed reductive subgroup $\textbf{H} \subset \textbf{G}_{F}$ and a character $v:\textbf{H}\longrightarrow \mathbb{G}_{m,F}$ such that they satisfy the following conditions $(*)$:

$\bullet$ The cocharacter $\mu:\mathbb{G}_{m,F}\longrightarrow \textbf{G}_{F}$ factor through $\textbf{H}$ and we still denote this cocharacter $\mathbb{G}_{m,F}\longrightarrow \textbf{H}$ by $\mu$.

$\bullet$ $v\circ \mu$ is the identity map for $\mathbb{G}_{m,F}$.

By the second condition, the map $v:\textbf{H}\longrightarrow \mathbb{G}_m$ is a quotient map (surjective). Moreover, taking $F$-points, we see that the induced map $\textbf{H}(F)\longrightarrow \mathbb{G}_{m}(F)=F^*$ is also surjective and if we equip $\textbf{H}(F)$ and $F^*$ with the induced $p$-adic topology, this map is an open map.

Consider the following conductor filtration on $\mathbb{G}_{m}(F)$: For any non-negative integer $m$, define \[ \mathbb{G}_m(m) = \begin{cases}  \bO_F^* & \mbox{if } m=0 \\ 1+\pi^{m}\bO_F & \mbox{if } m>0 \end{cases}. \]
 We will use $v$ to define a conductor filtration on $\textbf{H}(F)$, just define $H(m)=v^{-1}(\mathbb{G}_m(m))$. By the second condition, we get $\mu(\mathbb{F}_q^*) \subset H(0)$. Now we can translate the above lemma as this theorem:

\begin{thm}{\textbf{divisibility}}
\label{divisible}

$q-1 |$ $Hep_{\mu}$($\mu(\pi)$)$([1])$ in  $R[H(0)\backslash G /K]$.

\end{thm}

This divisibility can be translated as a kind of abstract relation. And it will also explain another reason behind the appearance of $q-1$.

Define $H^d=\ker (\textbf{H}(F) \overset{v}{\longrightarrow} \mathbb{G}_m(F) )$ and consider the module $R[H^{d} \backslash G/K]$, which will be a local analogue of our module of special cycles (see next section \ref{tamereal}). For an element $x\in H^{d} \backslash G/K$, we denote its corresponding element in $R[H^{d} \backslash G/K]$ by $[x]$. We have a distinguished element $[1]$.  Since $H^d$ is a normal subgroup of $\textbf{H}(F)$, the latter group acts on $H^d \backslash G/K$, and this action factors through $v$. Through this way we get an $\textbf{H}(F)$-action on $R[H^{d} \backslash G/K]$.

\begin{thm}{\textbf{abstract relation}}
\label{abstame}

The element $Hep_{\mu}(\mu(\pi))([1]) \in R[H^{d} \backslash G/K]$ lies in the image of the trace map

$Tr_{1,0}\overset{def}{=}Tr_{\frac{H(0)}{H(1)}}: R[H^{d} \backslash G/K]^{H(1)} \longrightarrow R[H^{d} \backslash G/K]^{H(0)}$.

\end{thm}

\begin{proof}

Because $H(0)$ fixes $[1]$ and the $H(0)$-action commutes with the operator $Hep_{\mu}(\mu(\pi))$, the element $Hep_{\mu}(\mu(\pi))([1])$ lies in $R[H^{d} \backslash G/K]^{H(0)}$.

Notice that each $H(0)$-orbit in $R[H^{d} \backslash G/K]$ is finite,  we obtain an $R$-linear isomorphism \[R[H(0)\backslash G/K]\cong R[H^{d}\backslash G/K ]^{H(0)}.\]

Denote the projection $H^{d}\backslash G/K \longrightarrow H(0)\backslash G/K$ by $pr_0$. For any element $C \in H(0)\backslash G/K$, this map is given by sending $[C]$ to $\displaystyle \sum_{x \in pr_0^{-1}(C)}[x]$.

Similarly, denote the projection $H^{d}\backslash G/K \longrightarrow H(1)\backslash G/K$ by $pr_1$, denote $H(1)\backslash G/K \longrightarrow H(0)\backslash G/K$ by $pr$. We get a $R$-linear isomorphism $R[H(1)\backslash G/K]\cong R[H^{d}\backslash G/K ]^{H(1)}$.

Consider the following commutative diagram:
\begin{diagram}
R[H(1)\backslash G/K] &\rTo^{\cong} &R[H^{d}\backslash G/K ]^{H(1)}\\
\dTo_{Tr} & &\dTo_{Tr_{1,0}}\\
 R[H(0)\backslash G/K] &\rTo^{\cong} &R[H^{d}\backslash G/K ]^{H(0)}
\end{diagram}
Here $Tr$ is the map induced by $Tr_{1,0}$, \textbf{not} the map induced by $pr$. Now let's explicit this map $Tr$:

For any $C \in H(1)\backslash G/K$, let $S(C)$ denote the cardinality of its stabilizer group in $\frac{H(0)}{H(1)}$. Then we have $S(C) | q-1$, because $q-1$ is the cardinality of $\frac{H(0)}{H(1)}$. And $Tr([C])=S(C)[pr(C)]$.

Now  in $ R[H(0)\backslash G/K]$, write $Hep_{\mu}(\mu(\pi))([1]) =\sum_{D} a_D[D]$, where $D$ runs over elements in $H(0)\backslash G/K$ with non-zero $a_D$. For each $D$, choose an element $C_D \in H(1)\backslash G/K$ such that $pr(C_D)=D$. By our divisibility theorem \ref{divisible}, for each $D$, $q-1| a_D$, thus $\frac{a_D}{S(C_D)}$ lies in $R$. Define $S_1=\sum_D \frac{a_D}{S(C_D)}[C_D] $, then $Tr(S_1)=Hep_{\mu}(\mu(\pi))([1])$. We're done.
\end{proof}
As in our previous lemma, we can also replace $[1]$ by $[t]$ ($t \in \textbf{T}(F)$) with the same proof.

Finally we propose a variant. This theorem is already enough for our main example in next section \ref{tamereal}. For more general cases, we only need a slight change, see section \ref{Gsp4} and section \ref{unitary}.

Suppose we have the following integral version conditions $(\spadesuit)$ of  conditions (*):

$\spadesuit$ Suppose there exists a closed reductive group $\textbf{H}\hookrightarrow \textbf{G}$ over $\spec(\bO_F)$ and a character $v:\textbf{H}\longrightarrow \mathbb{G}_m$ over $\spec(\bO_F)$ such that

$\bullet$ The cocharacter $\mu:\mathbb{G}_{m}\longrightarrow \textbf{G}$ factor through $\textbf{H}$ and we still denote this cocharacter $\mathbb{G}_{m}\longrightarrow \textbf{H}$ by $\mu$.

$\bullet$ $v\circ \mu$ is the identity map for $\mathbb{G}_{m}$.

And we will define the conductor filtration on $\textbf{H}(\bO_F)$ instead of the whole group $\textbf{H}(F)$. For any non-negative integer, define $H(m)=v^{-1}(\mathbb{G}_m(m))$. Next we replace the group $H^d$ by $H^{der}=\textbf{H}^{der}(F)$. Other things remain the same, then we  get the following version of abstract tame relation:

\begin{thm}{\textbf{variant abstract relation}}
\label{abstame2}

The element $Hep_{\mu}(\mu(\pi))([1]) \in R[H^{der} \backslash G/K]$ lies in the image of the trace map

$Tr_{1,0}\overset{def}{=}Tr_{\frac{H(0)}{H(1)}}: R[H^{der} \backslash G/K]^{H(1)} \longrightarrow R[H^{der} \backslash G/K]^{H(0)}$.

\end{thm}

\subsection{Realization}
\label{tamereal}

In this section, we will show that the abstract relation can be translated into "real" tame relations for special cycles. This section is in global setting, we will use notations in section \ref{gloset} again.

First we enlarge the coefficient ring $\mathbb{Z}$ into $R$. Here we take $R$ to be a $l$-adic integer ring. Recall we have taken a finite set $S_1$ ("bad primes") containing this prime $l$. Then for any prime $p \notin S_1$, we have $p^{-1}\in R$, so by the discussion about coefficients for Hecke algebra in section \ref{locset}, $R$ is large enough,  i.e. the coefficients of the Hecke polynomial belong the the Hecke algebra with coefficients in $R$.

Next we will use the previous section's theorems  to deduce a local result for our main example.

Recall that $F$ is a totally real field with a quadratic CM extension $F\longrightarrow E$. And $V$ is a $2n+1$-dimensional $F$-space  with a quadratic form $\phi$. Inside $(V,\phi)$, we have a $n$-dimensional $E$-hermitian space $(W,\psi)$. We have a pair of $F$-reductive groups $(\underline{\textbf{G}},\underline{\textbf{H}})=(SO(V),U(W))$ and our main example is $(\textbf{G},\textbf{H})=(\mathrm{Res}_{F/\mathbb{Q}}\underline{\textbf{G}},\mathrm{Res}_{F/\mathbb{Q}}\underline{\textbf{H}})$.

Take a finite place $v \notin S$ ("bad places") of $F$, suppose its underlying prime of $\mathbb{Q}$ is $p$, then $p \notin S_1$ and $v$ splits in $F\longrightarrow E$. Fix an isomorphism $ \displaystyle F \otimes \mathbb{Q}_p \cong \prod_{\mathfrak{p}|p}F_{\mathfrak{p}}$, for simplicity, we require the first factor to be $F_v$.

Consider the base change of our pair $(\textbf{G},\textbf{H})$ to $\mathbb{Q}_p$. The above isomorphism $ \displaystyle F \otimes \mathbb{Q}_p \cong \prod_{\mathfrak{p}|p}F_{\mathfrak{p}}$ will induce an isomorphism $\displaystyle \textbf{G}_{\mathbb{Q}_p} \cong \prod_{\mathfrak{p}|p}
\mathrm{Res}_{F_{\mathfrak{p}}/\mathbb{Q}_p}\underline{\textbf{G}}_{F_{\mathfrak{p}}}$, and for simplicity we define $\textbf{G}_p=\textbf{G}_{\mathbb{Q}_p}$, $\textbf{G}_{\mathfrak{p}}=\mathrm{Res}_{F_{\mathfrak{p}}/\mathbb{Q}_p}\underline{\textbf{G}}_{F_{\mathfrak{p}}}$.
Similarly we have $\displaystyle \textbf{H}_{\mathbb{Q}_p} \cong \prod_{\mathfrak{p}|p}
\mathrm{Res}_{F_{\mathfrak{p}}/\mathbb{Q}_p}\underline{\textbf{H}}_{F_{\mathfrak{p}}}$ and define $\textbf{H}_{p}=\textbf{H}_{\mathbb{Q}_p}$, $\textbf{H}_{\mathfrak{p}}=\mathrm{Res}_{F_{\mathfrak{p}}/\mathbb{Q}_p}\underline{\textbf{H}}_{F_{\mathfrak{p}}}$.

Extend the embedding $F \longrightarrow F_v$ into an embedding $\overline{\mathbb{Q}} \longrightarrow \overline{F_v}=\overline{\mathbb{Q}_p}$. Along this way, we pick up a prime $v^{+}$ of $E$ lying over $v$ and identify $E_{v^{+}}$ with $F_v$. Denote the conjugation of $v^{+}$ by $v^{-}$ and denote $E_{v^{-}}$ by $F_{\overline{v}}$. Under this isomorphism $E\otimes_{F}F_v\cong F_v \times F_{\overline{v}}$, the conjugation on the left side will correspond to swap factor involution on the right side.

Recall our Shimura cocharacter conjugacy class $[\mu_\textbf{G}]$ and now transfer it  into the $p$-adic setting.  The embeddings between fields $\mathbb{C}\hookleftarrow \overline{\mathbb{Q}} \hookrightarrow \overline{F_v}$ induce isomorphisms \[\Hom_{\mathbb{Q}}(F,\mathbb{C}) \cong \Hom_{\mathbb{Q}}(F,\overline{\mathbb{Q}}) \cong \Hom_{\mathbb{Q}}(F, \overline{F_v}),\] and conjugacy classes of cocharacters \[C_*(\textbf{G}_{\mathbb{C}}) \cong C_*(G_{\overline{\mathbb{Q}}})\cong C_*(G_{\overline{F_v}}).\] The isomorphism $\displaystyle \textbf{G}_p \cong \prod_{\mathfrak{p}}\textbf{G}_{\mathfrak{p}}$ induces an isomorphism $\displaystyle C_*(\textbf{G}_{p,\overline{F_v}})\cong \prod_{\mathfrak{p}} C_*(\textbf{G}_{\mathfrak{p},\overline{F_v}})$, and our conjugacy class of Shimura cocharacter $[\mu_\textbf{G}]$ will correspond to an element $[\mu_{\textbf{G},v}]=([\mu_v],1,...,1)$. Denote  $\textbf{G}_p(\mathbb{Q}_p)$ by $G_p$  with  associated Hecke algebra $He_{p}$ and  $\textbf{G}_{\mathfrak{p}}(\mathbb{Q}_p)$ by $G_{\mathfrak{p}}$ with associated Hecke algebra $He_{\mathfrak{p}}$. Through the natural isomorphism $\displaystyle He_p \cong \bigotimes_{\mathfrak{p}\mid p}He_{\mathfrak{p}}$, we have an identity for their Hecke polynomials $Hep_{\mu_{\textbf{G},v}}=Hep_{\mu_v}$. This finishes the first step reduction.

To apply previous section's theorems,  we next analyze our subgroup $\textbf{H}_v$. Recall $\textbf{H}_v=\mathrm{Res}_{F_v/\mathbb{Q}_p}\underline{\textbf{H}}_{F_v}$. Because $v$ is a split prime and we have fixed $E_{v^+}\cong F_v$, the unitary group $\underline{\textbf{H}}_{F_v}$ splits. More precisely, through $E \otimes_{F} F_v \cong F_v \times F_{\overline{v}}$, we have $W_{F_v}=W\otimes_{F}F_v\cong W_v \oplus W_{\overline{v}} $, where $W_v=W\otimes_{E}F_v$ and $W_{\overline{v}}=W \otimes_{E} F_{\overline{v}}$. Then we get $\underline{\textbf{H}}_{F_v} \cong GL(W_v)_{F_v}$. Similarly we have $\underline{\textbf{T}}_{F_v} \cong \mathbb{G}_{m,F_v}\times \mathbb{G}_{m,F_{\overline{v}}}$ and $\underline{\textbf{T}^1}_{F_v} \cong \mathbb{G}_{m,F_v}$. Under these isomorphisms, the determinant map $\det:\underline{\textbf{H}}_{F_v}\longrightarrow \underline{\textbf{T}^1}_{F_v}$ corresponds to the determinant map for general linear groups $\det:GL(W_v)\longrightarrow \mathbb{G}_{m,F_v}$, the  map $\underline{r}$  sends $(x,y)\in Al^*\times Al^*$ to $\frac{x}{y} \in Al^*$ for any $F_v$-algebra $Al$ and the inclusion $\underline{\textbf{T}^1}_{F_v}(Al)\hookrightarrow \underline{\textbf{T}}_{F_v}(Al)$ is $x \longrightarrow (x,x^{-1})$. Now take a reductive integral model $\underline{\mathcal{G}}_v$ over $\spec(\bO_{F_v})$ for $\underline{\textbf{G}}_{F_v}$ such that $\underline{\mathcal{G}}_v(\bO_{F_v})$ equals  $K_v$. Then take a Borel pair $(\underline{\mathcal{B}}_v,\underline{\mathcal{T}}_v)$ for $\underline{\mathcal{G}}_v$ so that $(\underline{\mathcal{B}}_{v,F_v}\bigcap \underline{\textbf{H}}_{F_v},\underline{\mathcal{T}}_{v,F_v} \bigcap \underline{\textbf{H}} _{F_v})$ is a Borel pair for $\underline{\textbf{H}}_{F_v}$.  This Borel pair for $\underline{\textbf{H}}_{F_v}$ will give us an ordered $F_v$-basis $(x_1,...,x_n)$ for $W_v$ so that $\underline{\mathcal{B}}_{v,F_v}\bigcap \underline{\textbf{H}}_{F_v}$ corresponds to upper triangular matrices and $\underline{\mathcal{T}}_{v,F_v} \bigcap \underline{\textbf{H}} _{F_v}$ corresponds to diagonal matrices.  Now define a cocharacter $\mu_{v,0}:\mathbb{G}_{m,F_v}\longrightarrow \underline{\textbf{H}}_{F_v}$ by sending $t$ to $(t,1,...,1)$ under this ordered basis. Composing it with the embedding $\underline{\textbf{H}}_{F_v}\hookrightarrow \underline{\textbf{G}}_{F_v}$, we get a cocharacter for $\underline{\textbf{G}}_{F_v}$ factoring through its maximal torus $\underline{\mathcal{T}}_{v,F_v}$ and this cocharacter has a unique extension to $\underline{\mathcal{T}}_{v}$ over $\spec(\bO_{F_v})$, for simplicity we still denote them by $\mu_{v,0}$.

Now we relate $\mu_{0,v}$ to $[\mu_v]$. Recall $\textbf{G}_v=\mathrm{Res}_{F_v/\mathbb{Q}_p}\underline{\textbf{G}}_{F_v}$, then over $\overline{F_v}=\overline{\mathbb{Q}_p}$, we have an isomorphism for conjugacy classes of cocharacters, \[  C_*(\textbf{G}_{v,\overline{F_v}})\cong \prod_{\delta \in Hom_{\mathbb{Q}_p}(F_v,\overline{F_v})} C_*(\underline{\textbf{G}}_{v,\delta,\overline{F_v}}),\] here $\underline{\textbf{G}}_{v,\delta,\overline{F_v}}= \underline{\textbf{G}}_{F_v}\otimes_{F_v,\delta}\overline{F_v}$ and for simplicity we assume the first factor corresponds to our fixed inclusion $F_v\hookrightarrow \overline{F_v}$. Then $[\mu_v]=([\mu_{v,0}],1,...,1)$. For $(\underline{\textbf{G}}_{F_v},[\mu_{v,0}])$ we have the Hecke polynomial $Hep_{\mu_{v,0}}\in \underline{He}_{v}[X]$, here $\underline{He}_v$ is its Hecke algebra $R[K_v \backslash \textbf{G}_{v}(F_v) / K_v]=He_{v}$, then $Hep_{\mu_{v,0}}=Hep_{\mu_v}$ by \cite{cornut2018} section 10.2.

Therefore we have finished the second reduction and ready to apply previous section's theorems now.

Because $\det(\mu_{v,0})$ is the identity map for $\mathbb{G}_{m,F_v}$, the pair $(\underline{\textbf{H}}_{F_v},\det)$ satisfies conditions (*) in previous section. The map $\det$ induces a conductor filtration  $\{\underline{\textbf{H}}_v(m)|m\geq 0\}$ on $\underline{\textbf{H}}_{F_v}(F_v)=\textbf{H}_v(\mathbb{Q}_p)$ and its kernel is exactly $\underline{\textbf{H}}^{der}_{F_v}(F_v)$. Denote this kernel by $H_v^{d}$ and define $H_{v}(m)=\underline{\textbf{H}}_{v}(m)$. Take a uniformizer $\pi_v \in \bO_{F_v}$  and apply  theorem \ref{abstame} to this pair $(\underline{\textbf{H}}_{F_v},\underline{\textbf{G}}_v)$, we get the following local result:
\begin{thm}{\textbf{local tame relation}}
\label{loctame}

The element $Hep_{\mu_{v,0}}(\mu_{v,0}(\pi_v))([1]) \in R[H_v^{d} \backslash G_v/K_v]$ lies in the image of trace map

$Tr_{v,1,0}\overset{def}{=}Tr_{\frac{H_v(0)}{H_v(1)}}: R[H_v^{d} \backslash G_v/K_v]^{H_v(1)} \longrightarrow R[H_v^{d} \backslash G_v/K_v]^{H_v(0)}$.

\end{thm}
Here $R[H_v^{d} \backslash G_v/K_v]$ is exactly the local version of our module for special cycles and $\mu_{v,0}(\pi_v)$  can be seen as an analogue of the \textbf{Frobenius}:

Still denote the image of $\mu_{v,0}(\pi_v)$ (resp. $\pi_v$) under the inclusion $\underline{\textbf{H}}(F_v)\hookrightarrow \textbf{H}(\mathbb{A}_f)$ (resp. $\underline{\textbf{T}^1}(F_v) \hookrightarrow \textbf{T}^1(\mathbb{A}_f) $) by $\mu_{v,0}(\pi_v)$ (resp. $\pi_v$). Here recall we have identified $\underline{\textbf{T}^1}_{F_v}$ with $\mathbb{G}_{m,F_v}$ and $\pi_v \in \mathbb{G}_{m,F_v}(F_v)$. Denote the class of $\pi_v$ in $\textbf{T}^1(\mathbb{A}_f)/\textbf{T}^1(\mathbb{Q})$ by $[\pi_v]$. Because $\det(\mu_{v,0}(\pi_v))=\pi_v=\underline{r}((...,1,\pi_v,1,...))$ and then $\mathrm{Art}_{E}^{1}([\pi_v])=Frob_{v^+}|_{E[\infty]}$. According to the reciprocity law, for any special cycle $Z_K(\widetilde{g})$, we have \[Frob_{v^{+}}(Z_K(\widetilde{g}))=Z_K(\mu_{v,0}(\pi_v)\widetilde{g}).\]

We also notice that the two conductor filtrations for torus (defined in section \ref{gloset} and section \ref{abstame}) are indeed the same. For any positive integer $m$, we have the following exact sequences:

\begin{diagram}
1 &\rTo &\mathbb{G}_{m}(\bO_{F_v}) &\rTo &\underline{\textbf{T}}(\bO_{F_v}) &\rTo^{\underline{r}}  &\underline{\textbf{T}}^1(\bO_{F_v}) &\rTo &1\\
   &  &\dTo^{red_0} &  &\dTo^{red_1} &  &\dTo^{red_2}\\
1 &\rTo &\mathbb{G}_{m}(\bO_{F_v}/\pi_v^m) &\rTo &\underline{\textbf{T}}(\bO_{F_v}/\pi_v^m)  &\rTo &\underline{\textbf{T}}^1(\bO_{F_v}/\pi_v^m) &\rTo &1
\end{diagram}

These three vertical reduction maps ($red_0$, $red_1$, $red_2$) are surjective, then the snake lemma implies $\underline{r}(\ker(red_1))=\ker(red_2)$. The zero step in both filtrations are $\underline{\textbf{T}}^{1}(\bO_{F_v})$. Therefore these two filtrations are the same. And in fact this proof holds for any prime of $F$.

Now recall our base cycle $z=Z_K(g)$ with $g=g_{S_1}\times g^{S_1}$ and recall we have the following natural map:
\[{\bigotimes_{w}}^{'} R[\underline{\textbf{H}}^{der}(F_w)\backslash \underline{\textbf{G}}(F_w)/K_w] \longrightarrow R[Z_K(\textbf{G},\textbf{H})].\] By our assumption for $S_1$ and $S$, we can assume $z=z_S \otimes 1^S$, here $\displaystyle 1^S=\bigotimes_{w \notin S}1$. For each prime $v \in \underline{\textbf{P}}$, by the above theorem, there exists an element $\widehat{z_v} \in R[\underline{\textbf{H}}^{der}(F_v) \backslash \underline{\textbf{G}}(F_v)/K_v]^{H_v(1)}$ such that $Tr_{v,1,0}(\widehat{z_v})=Hep_{\mu_{v,0}}(Frob_{v^+})([1])$. For any ideal $m \in \underline{\textbf{N}}$, we define a cycle $\displaystyle z(m)=z_S \otimes(\bigotimes_{v\mid m}\widehat{z_v})\otimes(\bigotimes_{v\nmid n,v \notin S}1)$. Putting everything together, we finally obtain the following result:
\begin{thm}{\textbf{tame relation}}

We have $z(1)=z$. And for any $m \in \underline{\textbf{N}}$, the special cycle $z(m)$ is defined over $\underline{E}[m]$, and for any $v \in \underline{P}$ does not divide $m$, we have $Tr_{\frac{\underline{E}[mv]}{\underline{E}[m]}}(z(mv))=Hep_{\mu_{\textbf{G},v}}(Frob_{v^+})(z(m))$.
\end{thm}
\begin{proof}
The special cycle $z(m)$ is fixed by $U^1(m)$ thus lies in $\underline{E}[m]$ and by definition $z(1)=z$. Notice that for any prime $w \notin S \bigcup\{v\}$, we have $z(m)_{w}=z(mv)_{w}$ ($w$-component). Applying the previous results, we get \[Tr_{\frac{\underline{E}[mv]}{\underline{E}[m]}}(z(mv)) \xlongequal{lemma\, \ref{locglo}} Tr_{\frac{U_{v}^1(0)}{U_v^1(1)}}(z(mv))\] \[\xlongequal{} z_S\otimes (Tr_{v,1,0}(\widehat{z_v})) \otimes (\bigotimes_{w \notin S \bigcup \{v\}}z(m)_{w})\] \[\xlongequal{theorem\, \ref{loctame}} z_S \otimes (Hep_{\mu_{\textbf{G},v}}(Frob_{v^+})([1])) \otimes (\bigotimes_{w \notin S \bigcup \{v\}}z(m)_{w})\] \[=Hep_{\mu_{\textbf{G},v}}(Frob_{v^+})(z(m)).\]
\end{proof}
Finally we make two remarks:

$\bullet$ Here for simplicity we take the base cycle $z=Z_K(g)$, it can be easily generalized to any cycle $\widetilde{z}=\sum_{i}Z_K(g_i)$. Just apply the above theorem  to each $Z_K(g_i)$, by linearity we're done.

$\bullet$ In this paper, we always use double coset description for Hecke algebra and use "formal" Hecke action. The geometric Hecke action is through the Geometric Hecke correspondence. These two actions are the same, we refer to Cornut's \cite{cornut2018} section 5.16 for details.

\section{Norm relation}

\subsection{Abstract relation}
\label{absnorm}

In this section, we come back to the local setting again.

Let $F$ denote a $p$-adic local field with a uniformizer $\pi$. Denote the cardinality of the residue field $\bO_F/\pi$ by $q$. Let $\textbf{G}$ denote a reductive group scheme over $\spec(\bO_F)$ and $\textbf{H}$ a closed subgroup scheme of $\textbf{G}$.  We assume this pair $(\textbf{G},\textbf{H})$ to be spherical over $\spec(\bO_F)$: there exists a Borel subgroup scheme $\overline{\textbf{B}}$ of $\textbf{G}$  such that the $\textbf{H}$-orbit of $[1]$ in $\textbf{G}/\overline{\textbf{B}}$ is open, equivalently, $Lie(\textbf{H})+Lie(\overline{\textbf{B}})=Lie(\textbf{G})$. Take a maximal torus $\textbf{T}$ inside $\overline{\textbf{B}}$ and denote by $\textbf{B}$ the Borel group of $\textbf{G}$ opposed to  $\overline{\textbf{B}}$ with respect to $\textbf{T}$.  Here we use notations compatible with those of  Loeffler's \cite{loeffler2019} section 4. Loeffler considers a more general situation where $\textbf{B}$ is replaced by a parabolic subgroup $Q_\textbf{G}$ and $[1]$ is replaced by $[u]$ ($u \in \textbf{G}(\bO_F)$). Our method can be generalized similarly but here this simplification is sufficient for our application.

 Denote $G=\textbf{G}(F)$, $H=\textbf{H}(F)$, $K=\textbf{G}(\bO_F)$ and $K_H=\textbf{H}(\bO_F)=H \bigcap K$. Denote the Hecke algebra for $G$ by $He=R[K \backslash  G/K]$, here $R$ is a suitable coefficient ring to realize Hecke polynomial (i.e. contains $\mathbb{Z}[q^{\pm\frac{1}{2}}]$). Now we will define a filtration on $G$ and $H$ ("level group filtration").

Denote the unipotent radical of $\textbf{B}$ and $\overline{\textbf{B}}$ by \textbf{N} and $\overline{\textbf{N}}$, so we have $\textbf{B}=\textbf{T} \ltimes \textbf{N}$ and $\overline{\textbf{B}} = \textbf{T} \ltimes \overline{\textbf{N}}$. Take  a strict $\textbf{B}$-dominant cocharacter $\mu \in X^+(\textbf{T})$, here strict means the associated parabolic group $\textbf{P}_{\mu}$ for $\mu$ is exactly $\textbf{B}$. Set $\tau=\mu(\pi)$ and thus we have $\tau \textbf{N}(O_F) \tau^{-1} \subset \textbf{N} (O_F)$, $\tau^{-1} \overline{\textbf{N}}(O_F) \tau \subset \overline{\textbf{N}}(O_F)$. Let $Hep_{\mu}\in He[X]$ and $U_{\mu}\in End_{\textbf{B}(F)}R[G/K]$ be the associated Hecke polynomial and $U$-operator.  Now for  any integer $m \geq 0$, define  $G_m=K \bigcap \tau^{m} K \tau^{-m}$, $H_m=G_m \bigcap \textbf{H}(F)$. So $G_0=K$, $H_0=K_H$. For a positive integer $m$, $G_m$ lies in the $\bO_F$-points of the "Big Bruhat cell" $\textbf{N}\times \textbf{T} \times \overline{\textbf{N}}$ and decomposes as $G_m=\tau^m\textbf{N}(\bO_F)\tau^{-m} \times \textbf{T}(\bO_F)\times \overline{\textbf{N}}(\bO_F)$. It follows that $G_m \supset G_{m+1}$ and $H_{m} \supset H_{m+1}$ for every $m$. Moreover, if $N_m = \textbf{N}(F) \bigcap G_m$, then $N_m = \tau^m \textbf{N}(\bO_F) \tau^{-m}$ and $N_m \hookrightarrow G_m $ induces a bijection $N_m/N_{m+i} = G_m / G_{m+i}$ for every $m$ and $i$.

 Regarding these level group filtrations,  we have the following comparison:
 \begin{lem}{\textbf{comparison} 1}
 \label{normlemma}

 For $m,i\geq 1$, the natural inclusion $H_m \subset G_m$ induces an isomorphism $H_m / H_{m+i}=G_m / G_{m+i}$.
 \end{lem}
\begin{proof}
This map is obviously injective and we only need to establish surjectivity.

Take an element $x \in G_m$ and denote its class in $(\textbf{G}/\overline{\textbf{B}})(\bO_{F})$ by $[x]$. Consider the reduction (or specialization) map $red:\textbf{G}(\bO_F)\longrightarrow \textbf{G}(\bO_F/\pi)$. Because $G_m=\tau^m\textbf{N}(\bO_F)\tau^{-m} \times \textbf{T}(\bO_F)\times \overline{\textbf{N}}(\bO_F)$, then $red(x) \in \overline{\textbf{B}}(\bO_F/\pi)$. Therefore $red([x])=red([1])$, here $red$ also denotes the reduction map for our flag scheme $(\textbf{G}/\overline{\textbf{B}})(\bO_F)\longrightarrow(\textbf{G}/\overline{\textbf{B}})(\bO_F/\pi)$.

Let $\Psi:\textbf{H}\longrightarrow \textbf{G}/ \overline{\textbf{B}}$ be the $\textbf{H}$-orbit map of $[1]$. By our spherical assumption ($Lie(\textbf{H})+Lie(\overline{\textbf{B}})=Lie(\textbf{G})$), this map is a smooth map. Because $red([x])=red([1])=\Psi(1)$ and $\bO_F$ is  Henselian, there exists an element $y \in \textbf{H}(\bO_F)$ such that $\Psi(y)=[x]$ and $red(y)=1$ (here $red$ is for $\textbf{H}$). Then $x=y \cdot c$, where $c \in \overline{\textbf{B}}(\bO_F)$. Since $c \in G_m$, it follows that $y$ belongs to $H_m$ and $[y]= [x]$  in $G_m/G_{m+i}$.
\end{proof}
Therefore we can identify $H_{m}/H_{m+i}$ with $N_{m}/N_{m+i}$ through these isomorphisms.

 Now consider our module $R[G/K]$. Denote $x_0=[1]=1_{K}$ and  define  $x_m=[\tau^m]=1_{\tau^m K}$ ($m$ is any positive integer).

 For any integer $m>0,i\geq 0$, we have the following lemma:
 \begin{lem}{\textbf{comparison} 2}

 $Tr_\frac{H_m}{H_{m+i}}(x_{m+i})=U_{\mu^{i}}(x_m)$ in $R[G/K]$, here $\displaystyle Tr_\frac{H_m}{H_{m+i}}(x_{m+i}) \overset{def}{=}\sum_{\delta \in H_{m}/H_{m+i}}\delta(x_{m+i})$.
 \end{lem}
\begin{proof}

If $i=0$, both operator are the identity. For positive integer $i$, it is a corollary of the above lemma \ref{normlemma}.

 Through the identifications $H_{m}/H_{m+i}=N_{m}/N_{m+i}=\tau^{m}(N_0/N_i) \tau^{-m}$, we find that  \[Tr_\frac{H_m}{H_{m+i}}(x_{m+i})=\sum_{\delta \in N_m/N_{m+i}}\delta(x_{m+i})=\sum_{z \in \frac{\textbf{N}(\bO_F)}{\tau^{i}\textbf{N}(\bO_F)\tau^{-i}}}\tau^mz\tau^{i}=U_{\mu^{i}}(x_m),\] where the last equality comes from the explicit formula (\ref{reda1}) for $U$-operator, there $I^{+}=\textbf{N}(\bO_F)$.
\end{proof}
 Now suppose the Hecke polynomial $Hep_{\mu}=\sum_i A_i X^{i}$, where $A_i \in R[K \backslash G/K]$ (Hecke algebra). Since  $U_{\mu^{i}}=(U_{\mu})^{i}$, applying Boumasmoud's relation (theorem \ref{reda2}), we have the following relation:

 \begin{lem}
 $\sum_i A_i Tr_{\frac{H_{m}}{H_{m+i}}}(x_{m+i})=0$ in $R[G/K]$.
 \end{lem}

 Now recall the local analogue of our modules for special cycles is $R[\textbf{H}^{der}(F)  \backslash G/K]$. Denote $\textbf{H}^{der}(F)$ by $H^{der}$, through the natural projection $G/K  \longrightarrow H^{der} \backslash G/K$, the above relation also holds in the second module. More precisely, define $H^{ab}=\frac{H}{H^{der}}$, $H_{m}^{ab}=Im(H_m)$, each fiber of the  map $H_{m}/H_{m+i} \longrightarrow H_{m}^{ab}/H_{m+i}^{ab}$ has constant cardinality, denote this cardinality by $c(m,i)$. Then we get:

  \begin{thm}{\textbf{abstract norm relation}}
 \label{absnormrel}

  $\sum_i c(m,i)A_i Tr_{\frac{H_m^{ab}}{H_{m+i}^{ab}}}(x_{m+i})=0 $ in $R[H^{der}  \backslash G/K]$.
 \end{thm}

  Finally we make a remark. For many spherical pairs $(\textbf{G},\textbf{H})$, we find that  $H_m=H \bigcap \tau^{m}K\tau^{-m}$. In other words, the $H$-stabilizer of $x_m \in G/K$  also stabilizes $x_0 \in G/K$. We conjecture that this holds in the general setup of this section (\textbf{stabilizer conjecture}) and verify it in many cases in appendix \ref{conjecture}. In particular, it holds for our main example. This observation will be used in next section \ref{normreal} for the calculation of $c(m,i)$.

 \subsection{Realization}
 \label{normreal}

We return  to the  global setting again and to the notations in section \ref{gloset}.

 Recall that we have chosen a special cycle $z=Z_K(g) \in R[Z_K(\textbf{G},\textbf{H})]$ and $z$ is defined over  the field $\underline{E}[1]$, which is the fixed field by $U^1(1)$. Choose a  prime $v$ of $F$ that doesn't belong to $S$ and $v \nmid disc(W), v \nmid disc(V)$ (discriminant of $W$ and $V$).

 First we verify the spherical condition over $v$ to apply previous section's theorems.

 Denote $\bO_{E_v}=\bO_E \otimes _{\bO_F} \bO_{F_v}$. Because $v$ is an unramified prime, there exists an element $\eta \in \bO_{E_v}^*$ such that $Tr(\eta)=\eta+\overline{\eta}=0$. Denote $(V_v,\phi_v)=(V,\phi)\otimes_{F}F_v$, $(W_v,\psi_v)=(W,\psi)\otimes_{F}F_v$. Because $v \nmid disc(W)$ and the unitary group $\underline{\textbf{H}}_{F_v}$ is unramified, there exists an orthogonal $E_v$-basis $\{w_1,...,w_n\}$ for $W_{F_v}=W \otimes_{F}F_v$, such that $\phi_{v}(w_i,w_i) \in \bO_{F_v}^*$. Suppose $V_v=W_v \perp F_v v_n$, we can assume $\phi_v(v_n,v_n) \in \bO_{F_v}^*$ due to $v \nmid disc(V)$. Because $Norm(\bO_{E_v}^*)=\bO_{F_v}^*$, we can rescale $\{w_1,...,w_n\}$ so that $\phi(w_n,w_n)+\phi(v_n,v_n)=0$, $\phi(w_i,w_i)+\phi(\eta w_{i+1},\eta w_{i+1})=0$ ($1\leq i \leq n-1$). Define $v_{i-1}=\eta w_{i}$ ($1\leq i \leq n$), we get an  orthogonal $F_v$-basis for $V_v$ \[\mathcal{B}_v=(v_0,w_1,v_1,w_2,...,w_n,v_n).\]
  Such basis is called special basis in Cornut's \cite{cornut2011} section 5.1.5. It defines two orthogonal decomposition of $V_v$, \[V_v=E_vw_1\perp ...\perp Ew_n\perp F_v v_n=F_v v_0 \perp H_1\perp...\perp H_n,\] where $H_i=F_v w_i \perp F_v v_i$ is an hyperbolic $F_v$-plane whose isotropic $F_v$-lines are spanned by $e_{\pm i}=\frac{1}{2}(v_i \pm w_i)$. Consider the ordered basis $(e_n,...,e_1,v_0,e_{-1},...,e_{-n})$, it defines an embedding $\underline{\textbf{G}}_{F_v} \hookrightarrow GL(V_v)_{F_v}$ and defines a Borel pair $(\underline{\textbf{Ts}}_v,\underline{\textbf{B}}_v)$ for $\underline{\textbf{G}}_{F_v}$, here $\underline{\textbf{B}}_v$  corresponds to upper triangular matrices. Consider the $\bO_{F_v}$-lattice $L_v$ spanned by the above basis, it is a self-dual lattice and extends our embedding over $\spec(\bO_{F_v})$,  $SO(L_v)\hookrightarrow GL(L_v)$. Through this embedding, we get a reductive integral model $\underline{\mathcal{G}}_v$ for $\underline{\textbf{G}}_{F_v}$ with the associated  Borel pair $(\underline{\mathfrak{Ts}}_v, \underline{\mathfrak{B}}_v)$. Consider the $\bO_{E_v}$-lattice $\widetilde{L}_w$ spanned by $(w_1,...,w_n)$, this is a self-dual lattice and extends the embedding $\underline{\textbf{H}}_{F_v} \hookrightarrow \underline{\textbf{G}}_{F_v}$ over $\spec(\bO_{F_v})$, $U(\widetilde{L}_w) \hookrightarrow SO(L_v)$, and we get a reductive integral model $\underline{\mathcal{H}}_v$ for $\underline{\textbf{H}}_{F_v}$. It is easy to check $\underline{\mathcal{H}}_v \bigcap \overline{\underline{\mathfrak{B}}_v}=1$ ($\overline{\underline{\mathfrak{B}}_v}$ is the Borel group opposite to $\underline{\mathfrak{B}}_v$ with respect to $\underline{\mathfrak{Ts}}_v$) and then $(\underline{\mathcal{H}}_v, \underline{\mathcal{G}}_v, \overline{\underline{\mathfrak{B}}_v})$ satisfies the spherical condition by dimension reason.

  Here we make some remarks.  In fact, since the intersection of $\underline{\mathcal{H}}_v$ and  $\overline{\underline{\mathfrak{B}}_v}$ is trivial,  the action of $\underline{\mathcal{H}}_v(\bO_{F_v})$ on the open orbit of $[1]$ in $\underline{\mathcal{G}}_v/\overline{\underline{\mathfrak{B}}_v} (\bO_{F_v})$ is transitive. In other words, any Borel group $\widetilde{B}_v \subset \underline{\mathcal{G}}_v$ with $Lie(\underline{\mathcal{H}}_v)+Lie(\widetilde{B}_v)=Lie(\underline{\mathcal{G}}_v)$ can be constructed as above.

 Now we can apply previous section's results. Take a uniformizer $\pi_v$ for $\bO_{F_v}$ and choose a strict $\underline{\mathfrak{B}}_v$-dominant cocharacter $\mu_v \in X_* ^+(\underline{\mathfrak{Ts}}_v)$ in the form of $(s_n,...,s_1,0,-s_1,...,-s_n)$ where $s_i$ are integers with $0<s_1<...<s_n$.  Still define $G_v=\underline{\textbf{G}}(F_v)$ with the hyperspecial group $K_v=\underline{\mathcal{G}}(\bO_{F_v})$.  Notice that we already chose a global level group $K$ with $v$-component "$K_v$". We can require these two notations defines the same group to avoid notation confusion. This can be done either by modifying the  $v$-component of the global level group $K$ or by changing our local integral model. Such slight change influences nothing. Let $\tau=\mu_v(\pi_v)$ and define $x_m=[\tau^m]=1_{\tau^mK_v} \in G_v/K_v$ for any non-negative integer $m$. Apply theorem \ref{absnormrel}, we get the following relation
 \begin{equation}
 \label{middlenorm}
 \sum_i c(m,i)A_i Tr_{\frac{\underline{\textbf{T}}^{1}_{v,m}}{\underline{\textbf{T}}^{1}_{v,m+i}}}(x_{m+i})=0 \in R[H_v^{der}  \backslash G_v/K_v],
 \end{equation}
 here $H_v^{der}=\underline{\textbf{H}}^{der}_{v}(F_v)$ and $\underline{\textbf{T}}^1_{v,m}=\underline{\textbf{H}}_{v,m}^{ab}$ with other notations defined in the previous section.

 Now we describe these constants $c(m,i)$. We will apply Cornut's explicit computation in \cite{cornut2011}. That paper works under inert assumption therefore from now on we assume $v$ to be inert. By Cornut's computation in \cite{cornut2011} section 5.1.14, for any non-negative integer $m$, we have $\underline{\textbf{T}}^1_v(m)=\det(\underline{\textbf{H}}_{v,m}^{ab})=U_v^1(ms_1)$, here $U_v^1(ms_1)$ is the conductor filtration defined in global setting \ref{gloset}. For the benefit of readers, we make some remarks about this computation. The case $m=0$  is trivial, we only need to care about $m>0$.  We use similar notation (special basis etc) as Cornut's \cite{cornut2011}. The related computation can be found in \cite{cornut2011} section 5.1.5-5.1.14. He dealt with more general cases there but we only need to use a very good case. The elements $x_m$ correspond to hyperspecial vertices in the Bruhat-Tits building of $\underline{\textbf{G}}_{F_v}$. Cornut implicitly described this building via self-dual norms on the orthogonal space $(V_v,\psi_v)$. There he studied the conductor of the stabilizer of $x_m$ in $\underline{\textbf{H}}_{F_v}(F_v)$. By the stabilizer conjecture (see appendix \ref{conjecture}), this group is exactly the level group for   $\underline{\textbf{H}}_{F_v}(F_v)$ that we defined in previous section. In his notations, our element $x_m$ will correspond to a norm $\alpha$ determined by the $n$-tuple $(ms_1,...,ms_n)$. Because $0<ms_1<...<ms_n$, this $n$-tuple satisfies his condition $(\textbf{SP})$ (see \cite{cornut2011} section 5.1.6). And this $n$-tuple determines another $n$-tuple $(c_1,...,c_n) \in \mathbb{Z}^n$, $c_1=s_1$, $s_i=s_i+s_{i-1}$ ($2\leq i \leq n$). By \cite{cornut2011} section 5.1.9, we have $\omega(\alpha)=[s_1,...,s_n]_D$. Then the lemma in \cite{cornut2011} section 5.1.14 tells us $\det U(\alpha)=U_{r}$ (in his notation) with $r=c_1=ms_1$. In our notations, this group is exactly $U_v^1(ms_1)$.

 For any positive integers $m$ and $i$, recall we have a natural identification  $\underline{\textbf{H}}_{m,v}/\underline{\textbf{H}}_{m+i,v}=N_{m,v}/N_{m+i,v}$, here $N_{m,v}=\tau^{m}\mathcal{N}_v(\bO_{F_{v}})\tau^{-m}$ ($\mathcal{N}_v$ is the unipotent radical of $\underline{\mathfrak{B}}_v$).  Then its  cardinality (denote it by $y(m,i)$) only depends on $i$. Moreover, using the root group map for $\mathcal{N}_v$, we can also compute it explicitly. It is $ q_{F_v}^{\langle i \mu_v, 2\rho_v \rangle}$, here $q_{F_v}$ is the cardinality for the residue field $\bO_{F_v}/\pi_v$ and $\rho_v$ is the half sum of positive root of $\underline{\mathcal{G}}_v$. Now for each $m$, define the conductor $con(m)$ to be the smallest non-negative integer $cc$ with $\det(\underline{\textbf{H}}_{m,v})\supset U_v^1(cc)$. Here by the above computation, we already know $con(m)=ms_1$. We thus obtain the following  explicit formula for $c(m,i)$, \[c(m,i)=q_{F_v}^{i(\langle \mu_v,2\rho_v\rangle - s_1)}=q_{F_v}^{i(\sum_{j=1}^{n}(2n-2j+1)s_{n+1-j}-s_1)}.\] In fact, the key point is that our main example satisfies the following \textbf{conditions}:

 $\clubsuit$ For large enough integer $m$, the numbers $c(m,1)$ are constant.

 $\heartsuit$  The conductor $con(m)$ grows into infinity.

 We make some remarks. The condition $\heartsuit$ will cut out a non-trivial $p$-adic (suppose $v|p$) extension for our resulting special cycles. We have a general strategy to verify this condition, avoiding explicit calculation. For positive integer $m$, the image of  $G_{v,m}$ in $\underline{\mathcal{G}}_v(\bO_{F_v}/\pi_v^m)$  lies in $ \overline{\underline{\mathfrak{B}}_v}(\bO_{F_v}/\pi_v^{m})$ after  modulo $\pi_v^m$. Because we know $\underline{\mathcal{H}}_v \bigcap \overline{\underline{\mathfrak{B}}_v}=1$ ("small intersection"),   it follows that the image of $H_{v,m}$ in $\underline{\mathcal{G}}_v(\bO_{F_v}/\pi_v^m)$ is trivial.  Then its abelian quotient lies in $U_v^1(m)$, $con(m) \geq m$ and the condition $\heartsuit$ is satisfied. For general Shimura pair $Sh({\widetilde{H}})\longrightarrow Sh(\widetilde{G})$ we may not have such trivial intersection property, but we can use a similar argument, see $Gsp_4$ example in section \ref{Gsp4} and the similitude version unitary GGP pair in section \ref{unitary}. The condition $\clubsuit$ is needed to turn abstract norm relations into a norm compatible family (under ordinary condition). At present the author doesn't know how to verify it formally without explicit computation.  Here the large enough assumption on $m$ is a very mild assumption in the study of norm relations. It will simplify the study about $c(m,1)$ under condition $\heartsuit$. 

Due to condition $\clubsuit$, we define $C_1=c(m,1)$ and we have $C_1^{i}=c(m,i)$  and denote it by $C_i$. Now we construct a family of special cycles. Write the base cycle $z=z^v \otimes 1_v$ and  for any non-negative integer $m$, we define $z_{cy(m)}=z^v \otimes [\tau^m]$, $U^1(cg(m))=U^{1,v} \times \underline{\textbf{T}}^1_{v,m}=U^{1,v} \times U^1_v(ms_1)=U^1((v)^{ms_1})$ and denote the corresponding fixed field by $\widetilde{E}(cg(m))$. We have the following realization:

\begin{thm}{\textbf{norm relation}}

We have $z_{cy(0)}=z$. And for $m>0$, the cycle $z_{cy(m)}$ lies in the field $\widetilde{E}(cg(m))$, they satisfy \[\sum_i C_i A_i Tr_{\frac{\widetilde{E}(cg(m+i))}{\widetilde{E}(cg(m))}}(z_{cy(m+i)})=0.\]
\end{thm}
\begin{proof}

This is  a translation of above results.

 The first two statements are obvious. Notice  that $U^1(1) \bigcap \textbf{T}^1(\mathbb{Q})=1$, similar to tame relation case (lemma \ref{locglo}), we also have a local-global connection in norm case,  \[Gal(\widetilde{E}(cg(m+i))/ \widetilde{E}(cg(m)))= \frac{\underline{\textbf{T}}^1_{v,m}}{\underline{\textbf{T}}^1_{v,m+i}}.\]

 Now apply the above relation \ref{middlenorm}, we're done.
\end{proof}

    This norm relation can be used to get a norm compatible family (under ordinary condition). The following technique is standard. For example, in Heegner point case, see \cite{nek} section 1.5.

    We will state this common technique as a lemma. Fix an embedding $\iota_p:\overline{\mathbb{Q}}\hookrightarrow \overline{\mathbb{Q}_p}$. Suppose we have a polynomial $Pol \in \overline{\mathbb{Q}}[X]$ and denote the completion of its splitting field through $\iota_p$ by $PF$. Denote the integer ring of $PF$ by $\bO$ with a uniformizer $\pi$ and  suppose $Pol$ has a root $b \in \bO^*$ (\textbf{ordinary condition}). Suppose it decomposes as \[Pol=\sum_{i=0} ^{k} e_i X^{i}= (X-b)(\sum_{i=0}^{k-1}b_i X^{i}).\] Now suppose the Galois group $Gal_{E}$ acts on a $PF$-space $V_M$ continuously ($p$-adic Galois representation) with a Galois stable $\bO$-lattice $M$. Consider another Galois lattice $M^{'}=\frac{1}{\pi^{ck}}M$, here $ck$ is a suitable integer such that $b_i(M)\subset M^{'}$ and $e_i(M) \subset M^{'}$. Now for simplicity denote the Galois trace $Tr_{\frac{\widetilde{E}(cg(l))}{\widetilde{E}(cg(r))}}$ as $Tr_{l,r}$, then we have the following lemma:

    \begin{lem}{\textbf{translation}}

    Assume for any integer $m>M_0$ ($M_0$ is fixed), there are elements $\mathcal{Y}_{m} \in H^1(Gal_{\widetilde{E}(cg(m))}, M)$  such that  $\displaystyle \sum_{i=0}^{k}e_iTr_{m+i,m}(\mathcal{Y}_m)=0 \in H^1(Gal_{\widetilde{E}(cg(m))},M^{'})$, then take $\displaystyle \mathcal{X}_m=b^{-m}\sum_{i=0}^{k-1}b_iTr_{m+i,m}(\mathcal{Y}_{m+i})$ in the later group. We have $Tr_{m+1,m}(\mathcal{X}_{m+1})=\mathcal{X}_m$.
    \end{lem}
    \begin{proof}
    This is a routine check. Notice \[Tr_{m+1,m}(b^{m+1}\mathcal{X}_{m+1})-b(b^{m}\mathcal{X}_m)=  \sum_{i=0}^{k-1}b_iTr_{m+i+1,m}(\mathcal{Y}_{m+1+i})-b\sum_{i=0}^{k-1}b_iTr_{m+i,m}(\mathcal{Y}_{m+i})\] \[=\sum_{i=0}^{k}e_i Tr_{m+i,m}(\mathcal{Y}_m)=0.\]   Because $b \in \bO^*$, we're done.
    \end{proof}

    Finally we make some remarks. In application of this lemma, we will use cohomology theory. See section \ref{coh} for more details. The rational $p$-coefficient etale cohomology  of our ambient Shimura variety $Sh_\textbf{G}(K)$ will provide the $p$-adic Galois representation $V_M$ with commutative Hecke action and its integral $p$-adic etale cohomology will provide the Galois-Hecke stable lattice $M$. Because the $v$-component $K_v$ of $K$ is hyperspecial, the associated automorphic representation will be unramified at $v$ and the local Hecke algebra over $v$ will acts on $V_M$ as scalar, determined by Satake parameters. And these scalars for our Hecke operator $A_i$ belong to $\overline{\mathbb{Q}}$ due to strong $C$-arithmetic properties (much stronger than what we need), see \cite{shin} Proposition 2.15 for more details. In many good cases, we can also expect these numbers belong to $\bO$ (integrality properties). Then we can just take $M^{'}=M$, i.e. we do not need to rescale the lattice. The next step is to  use $p$-adic Abel-Jacobi map to realize our cycles as elements in the Galois cohomology group and apply the above lemma. Here we first need to make these cycles cohomological trivial, again we refer to section \ref{coh} for more details.

\section{Other examples and further development}

\subsection{$Gsp_4$ example}
\label{Gsp4}

We will consider the following  embedding:

 $\textbf{H}=GU(1)\times_{G_m}GL_2 \longrightarrow \textbf{G}=Gsp_4$ (where $GU(1)=\mathrm{Res}_{E/\mathbb{Q}} \mathbb{G}_{m,E}$ and $E$ is an imaginary quadratic field).

Consider the 4-dim $\mathbb{Q}$-space $V$ with a symplectic form $J$, we choose a suitable ordered basis $(e_1,e_2,e_3,e_4)$ for $V$ so that $J$ corresponds to the matrix \[\begin{pmatrix}
0  & 0 & 0 & 1\\
0 & 0 & 1 & 0\\
0 & -1 & 0 & 0\\
-1 & 0 & 0 & 0
\end{pmatrix}.\]
For simplicity, we will use $E=\mathbb{Q}(\sqrt{-1})$ as an example, the general case is similar. Then $GU(1)$ can be embedded into $GL_2$, sending $a+b\sqrt{-1}$ to matrix $\left(\begin{smallmatrix}
a & b\\
-b & a
\end{smallmatrix}\right)$.

Now we can embed \textbf{H} into \textbf{G}, it will correspond to those matrices (where $a^2+b^2=xw-zy$) inside $Gsp_4$:

\[\begin{pmatrix}
a  & 0 & 0 & b\\
0 & x & y & 0\\
0 & z & w & 0\\
-b & 0 & 0 & a
\end{pmatrix}.\]

In fact, this group embedding has a factorization $\textbf{H}\longrightarrow \textbf{M}=GL_2 \times_{\mathbb{G}_m} GL_2 \longrightarrow \textbf{G}$. The embedding $\textbf{M}\longrightarrow \textbf{G}$ is obtained by replacing $\left(\begin{smallmatrix}
a & b\\
-b & a
\end{smallmatrix}\right)$ with $\left(\begin{smallmatrix}
l & p\\
q & r
\end{smallmatrix}\right)$ in the above expression.

These group maps will induce maps between Shimura varieties. Recall the  torus $GU(1)$ defines a Shimura variety, the map from Deligne torus $\mathrm{Res}_{\mathbb{C}/\mathbb{R}}\mathbb{G}_m(\mathbb{R})$ to $GU(1)(\mathbb{R})$  is an identification, sending $x+y\sqrt{-1}$ to $x+y\sqrt{-1}$ ($x,y \in \mathbb{R}$). And  the inclusion $GU(1) \longrightarrow GL_2$ will induce a map between their Shimura varieties. This is the usual map defining a  CM point inside the modular curve. Similarly, we have maps $Sh(\textbf{H})\longrightarrow  Sh(\textbf{M})\longrightarrow Sh(\textbf{G})$.

Loeffler, Skinner and Zerbes  have used the embedding $GL_2\times_{\mathbb{G}_m} GL_2 \longrightarrow Gsp_4$ to construct an Euler system (see \cite{loeffler2017}). The first embedding  $\textbf{H} \longrightarrow \textbf{M}$ can be seen as a kind of  "base change" of Heegner points. Roughly speaking, we can think of $Sh(\textbf{H})$ as a family of Heegner points or product of Heegner points with the modular curve.

What's more, for $E=\mathbb{Q}(i)$, this pair can seen as a kind of lift for our pair $(U(1,1),SO(3,2))$. Consider the standard representation $V$ for $Gsp_4$  and let  $\rho$ be the induced  representation of $Gsp_4$ on the 6-dimensional space $\wedge ^2 V$. Twist $\rho$ by $v^{-1}$, where $v$ is the standard 1-dimensional character defining $Gsp_4$, $J(g(x),g(y))=v(g)J(x,y)$. There is a  quadratic form on $\wedge ^2 V$ defined by the wedge product $\wedge ^2 V \times \wedge^2 V \longrightarrow \wedge^4 V \cong \mathbb{Q}$. Then the representation $\rho v^{-1}$ is an orthogonal representation for $Gsp_4$. And it fixes the line generated by $e_1 \wedge e_4+e_2 \wedge e_3$. Then the orthogonal complement $L$ for this line will provide a 5-dimensional orthogonal representation of $Gsp_4$. Through this way, we get a map $Gsp_4 \longrightarrow SO(3,2)$. Its kernel is $\mathbb{G}_m=Z(Gsp_4)$ (center of $Gsp_4$).

Then we need to construct a 2-dimensional $E$-hermitian space inside this 5-dimensional quadratic $\mathbb{Q}$-space. This is equivalent to define a "complex" structure (linear map $S$ with $S^2=-1$) compatible with the quadratic form ($S$ is orthogonal) on a suitable 4-dimensional subspace $W$. We choose $W$ to be the space generated by $\{e_1 \wedge e_2, e_2 \wedge e_4, e_3 \wedge e_4, e_1 \wedge e_3\}$.

Consider the following central element $S$ of $\textbf{H} \subset \textbf{G}$:
\[\begin{pmatrix}
0 & 0 & 0 & -1\\
0 & -1 & 0 & 0\\
0 & 0 & -1 & 0\\
1 & 0 & 0 & 0
\end{pmatrix}.\] Then $S$ also acts on $\wedge^2 V$ and we can check that  $S$ stabilizes   $W$ and induces the desired hermitian structure.

Consider the conjugation of $S$ on $Gsp_4$ (this is a four-order automorphism), its fixed subgroup is exactly $\textbf{H}=GU(1)\times_{\mathbb{G}_m} GL_2$. Moreover, here is an additional interesting fact. The fixed subgroup for conjugation by $S^2$ is exactly $\textbf{M}=GL_2 \times_{\mathbb{G}_m} GL_2$. Therefore we get two symmetric pairs, $\textbf{M}\longrightarrow \textbf{G}$ and $\textbf{H}\longrightarrow \textbf{M}$. Because \textbf{H} commutes with $S$, thus its action on $\wedge^2 V$ also commutes with $S$. Then $W$ is an $\textbf{H}$ representation, and we get the map $\textbf{H}\longrightarrow U(W)$. This is a surjection with kernel  $\mathbb{G}_m$. We have the following commutative diagram:

\begin{diagram}
\textbf{H}=GU(1) \times_{\mathbb{G}_m} GL_2 & \longrightarrow & \textbf{G}=Gsp_4\\
\downarrow &  & \downarrow\\
U(1,1) & \longrightarrow & SO(3,2)
\end{diagram}

Moreover, it induces  maps between Shimura datum, thus we can think of $Sh(\textbf{H}) \longrightarrow Sh(\textbf{G})$ as a kind of lift for $Sh(U(1,1))\longrightarrow Sh(SO(3,2))$. This lift allows us to replace the abelian type Shimura variety $SO(3,2)$ by the much more familiar Siegel threefold, and thus opens the way to a study of these special cycles  using PEL type moduli spaces.

Now we  study this pair in the framework of this paper.

We first compute the stabilizer involved in the parametrization lemma \ref{para}.  It is well known that the associated Hermitian symmetric domain for the Siegel modular variety is the Siegel half-space. More precisely, we can  rearrange the chosen ordered basis $(e_1,e_2,e_3,e_4)$ into $(e_1,e_2,e_4,e_3)$ so that the symplectic form $J$ is $\left(\begin{smallmatrix}
0 & I_2 \\
-I_2 & 0
\end{smallmatrix}\right)$. Then $Gsp_4(\mathbb{R})$  will acts on the Siegel  half-space $\mathbb{H}_2=\{Z \in M_2(\mathbb{C}),Z^{T}=Z, Im(Z)>0\}$,  the matrix $\left(\begin{smallmatrix}
A &  B\\
C & D
\end{smallmatrix}\right)$ sends $Z$ into $(AZ+B)(CZ+D)^{-1}$. This is the hermitian symmetric domain $X$ for $Sh(\textbf{G})$. The hermitian symmetric domain $Y$ for $Sh(\textbf{H})$ is the positive half plane, it is already connected (unlike $GL_2$ case). Inside $X$, $Y$ will correspond to $\left(\begin{smallmatrix}
\sqrt{-1} & 0 \\
0 & z
\end{smallmatrix}\right)$, where $z=a+b\sqrt{-1}$($b>0$) lies in the positive half plane. A direct argument  shows that $Stab_{\textbf{G}(\mathbb{Q})}(Y)=\textbf{H}(\mathbb{Q})$. Choosing a suitable neat level group $K\subset \textbf{G}(\mathbb{A}_f)$, we obtain a similar parametrization:
\[Z_K(\textbf{G},\textbf{H})=\textbf{H}(\mathbb{Q}) \backslash \textbf{G}(\mathbb{A}_f)/K.\]

Now let's consider the reciprocity law for $\pi_0(Sh(\textbf{H}))$. The derived subgroup $\textbf{H}^{der}=SL_2$ is simply connected, and the abelian quotient map is $pr:\textbf{H}=GU(1)\times_{\mathbb{G},m} GL_2\longrightarrow\textbf{H}^{ab}= GU(1)$. Then $E$ is the common reflex field for $Sh(\textbf{H})$ and $Sh(\textbf{H}^{ab})$ and the reciprocity law for $Sh(\textbf{H}^{ab})$ is the identity map, $\mathrm{Res}_{E/\mathbb{Q}}\mathbb{G}_m \overset{=}{\longrightarrow} GU(1)$. As in our main example, we also have $\overline{\textbf{H}(\mathbb{Q})}=\textbf{H}(\mathbb{Q})\textbf{H}^{der}(\mathbb{A}_f)$ inside $\textbf{H}(\mathbb{A}_f)$. Here we use a special property about imaginary quadratic field,  $\textbf{H}^{ab}(\mathbb{Q})$ is closed (in fact discrete) in $\textbf{H}^{ab}(\mathbb{A}_f)$, this is not true for more general CM fields.

Now everything is similar to our main example, we have the Hecke action on the right, and the Galois action through left multiplication by $\textbf{H}(\mathbb{A}_f)$ via reciprocity law. Writing $Z_K(\textbf{G},\textbf{H})=\textbf{H}(\mathbb{Q})\textbf{H}^{der}(\mathbb{A}_f) \backslash \textbf{G}(\mathbb{A}_f)/K$, we also have the following natural map: \[\bigotimes_{p} {^{'}}\mathbb{Z}[\textbf{H}^{der}(\mathbb{Q}_p)\backslash \textbf{G}(\mathbb{Q}_p)/K_p]\longrightarrow \mathbb{Z}[Z_K(\textbf{G},\textbf{H})].\]  Therefore we can apply our general strategy to this pair now.

 We first introduce a one dimensional torus quotient of $\textbf{H}^{ab}$ to define conductor filtrations. Define $\textbf{T}^1=U(1)=\ker(GU(1)\overset{Norm}{\longrightarrow}\mathbb{G}_{m,\mathbb{Q}})$ and consider the map  $r:GU(1)\longrightarrow \textbf{T}^1$ over $\mathbb{Q}$ by sending $z$ to $\frac{z}{\overline{z}}$. Then for tame relations, over a split prime $p$, we will use this map $r(pr):\textbf{H}\longrightarrow \textbf{T}^1$ to define conductor filtrations and apply theorem \ref{abstame2} to get abstract tame relations. Follow the  arguments in section \ref{tamereal} (realization of tame relations), we can produce special cycles for $Sh_\textbf{G}(K)$  with the desired tame relations.

 For norm relations, the first task is to  construct s suitable Borel subgroup to establish the spherical condition for this pair.

Here we  work over the global field $\mathbb{Q}$ directly.

Recall we have chosen a suitable ordered basis $(e_1,e_2,e_3,e_4)$ so that J corresponds to \[\begin{pmatrix}
0  & 0 & 0 & 1\\
0 & 0 & 1 & 0\\
0 & -1 & 0 & 0\\
-1 & 0 & 0 & 0
\end{pmatrix}.\]Then consider a new ordered basis $(\gamma_4,\gamma_3,\gamma_2,\gamma_1)=(e_1+e_2,e_3-e_4,e_1-e_2,e_3+e_4)$. This ordered basis defines a Borel pair  $(\textbf{T},\overline{\textbf{B}})$ for $\textbf{G}$ over $\mathbb{Q}$. And we will use $\textbf{B}$ to denote the  Borel subgroup of $\textbf{G}$ opposite to $\overline{\textbf{B}}$ with respect to $\textbf{T}$.  We have $\textbf{H} \bigcap \overline{\textbf{B}}=\mathbb{G}_{m,\mathbb{Q}}=Z(\textbf{G})$. For all but finitely many primes $p$ for $\mathbb{Q}$, we can extend these groups over $\spec(\mathbb{Z}_p)$, and denote $(\mathcal{H}_p,\mathcal{G}_p, \overline{\mathcal{B}}_p, \mathcal{T}_p, \mathcal{B}_p)$ for the  integral models of $(\textbf{H}_{\mathbb{Q}_p},\textbf{G}_{\mathbb{Q}_p}, \overline{\textbf{B}}_{\mathbb{Q}_p},\textbf{T}_{\mathbb{Q}_p},\textbf{B}_{Q_p})$. Then $(\mathcal{H}_p,\mathcal{G}_p)$ is a spherical pair (with respect to $\overline{\mathcal{B}}_p$) by dimension reason.

 Recall our torus $\textbf{T}^1=U(1)$ over $\mathbb{Q}$, we will still use it to define conductors for norm relations. Take a good unramified odd prime $p$ so that we have a reductive integral model $\mathcal{T}^1_p$ for $\textbf{T}^1_{Q_p}$ and extend the quotient maps $\textbf{H}\overset{pr}{\longrightarrow} GU(1)\overset{r}{\longrightarrow} \textbf{T}^1$ over $\spec(\mathbb{Z}_p)$ (still denote these maps by $pr,r$). From now on, we will work over $\mathbb{Q}_p$. We will do some explicit matrix calculations first.

 Under this ordered basis $(\gamma_1, \gamma _2, \gamma _3, \gamma_4 )$, the Borel subgroup $\textbf{B}_p$ will correspond to upper triangular matrices.  Take a strictly $\mathcal{B}_p$-dominant cocharacter $\mu=(a_1,a_2,a_3,a_4)  \in X_*^+(\mathcal{T}_p)$(where $a_1>a_2>a_3>a_4$ and $a_1+a_4=a_2+a_3$). Take  $\tau=\mu(\pi)$, where $\pi$ is a uniformizer for $\mathbb{Q}_p$, for example  $\pi=p$). The elements of $\textbf{H}(\mathbb{Q}_p)$ given by \[\begin{pmatrix}
a  & 0 & 0 & b\\
0 & x & y & 0\\
0 & z & w & 0\\
-b & 0 & 0 & a
\end{pmatrix}\] under the ordered basis $(e_1,e_2,e_3,e_4)$  corresponds to the elements of $\textbf{G}(\mathbb{Q}_p)$ given by \[\frac{1}{2}\begin{pmatrix}
w+a & -b-z & w-a & z-b\\
b-y & a+x & -y-b & a-x\\
w-a & b-z & w+a & b+z\\
b+y & a-x & y-b & a+x
\end{pmatrix}\] under the ordered basis $(\gamma_1, \gamma _2, \gamma _3, \gamma_4)$.  Then it is easy to see that $\textbf{H}(\mathbb{Q}_p)\bigcap \tau^{m} \mathcal{G}_p(\mathbb{Z}_p) \tau^{-m}$ lies in $\mathcal{H}_p(\mathbb{Z}_p)$, in other words, the stabilizer conjecture holds.  Then  for any non-negative integer $m$, we define $H_{p,m}=\textbf{H}(\mathbb{Q}_p)\bigcap \tau^{m} \mathcal{G}_p(\mathbb{Z}_p) \tau^{-m}$ as in section \ref{absnorm}. The next task is to compute the numbers $c(m,i)$ involved in the abstract norm relations (theorem \ref{absnormrel}).

 Consider the conductor filtration on $\mathcal{T}^1_p(\mathbb{Z}_p)$. As usual, $T^1_p(m)=\{\frac{z}{\overline{z}},z \in \mathbb{Z}_p+\pi^m\bO_{E_p}\}$. And we similarly define conductors for $H_{p,m}$, $con(m)$ is defined to be the minimal non-negative integer $cc$ with $r(pr)(H_{p,m})\supset T^1_p(m)$.

 We first make a formal argument. Because the intersection $\textbf{H} \bigcap \overline{\textbf{B}}=\mathbb{G}_m$ lies in the kernel of $r(pr)$, the conductors $con(m)$ grows to infinity, just like in our main example. Thus condition $\heartsuit$ (in section \ref{normreal}) is satisfied.

 To check condition $\clubsuit$ (in section \ref{normreal}), we will explicit these numbers $c(m,i)$. As in our main example, this involves computing conductors.

 For any positive integer $m$, consider the elements in $H_{p,m}$, through their matrices under the ordered basis $(\gamma_1,\gamma_2,\gamma_3,\gamma_4)$. We find $v_p(a)=0$ and $v_p(b)\geq m(a_1-a_2)$. Thus $pr(H_{p,m})\subset W_m=\mathbb{Z}_p^*+\sqrt{-1}\pi^m\mathbb{Z}_p$. Let us now see that this is actually a surjection, i.e. $pr(H_{p,m})=W_m$.

 For any $a \in \mathbb{Z}_p^*$ and $b \in p^{m}\mathbb{Z}_p$, consider this element $x_{a,b} \in \textbf{H}(\mathbb{Q}_p) $ given by \[\begin{pmatrix}
 a &0 &0 & b\\
 0 &a &-b & 0\\
 0 &b &a  &0\\
 -b &0 &0 &a \end{pmatrix}\] under the ordered basis $(e_1,e_2,e_3,e_4)$. Computing $x_{a,b}$ in the ordered basis $(\gamma_1,\gamma_2,\gamma_3,\gamma_4)$, we find that $x_{a,b}\in H_{p,m}$.  Notice that $pr(x_{a,b})=a+b\sqrt{-1}$, thus $pr(H_{p,m})=W_m$.

 Then the conductor $con(m)$ equals  $m(a_1-a_2)$ and  condition $\clubsuit$ holds. And we have the following explicit formula for $c(m,i)$: \[c(m,i)=p^{i(2a_1+a_2-2a_3-a_4)}.\]
   Thus we can apply the same idea in section \ref{normreal} to construct special cycles with norm relations along the anticyclotomic extension.

 Finally we make a remark. We have another "dual" version embedding. We can swap factors to get $\overline{\textbf{H}}=GL_2 \times_{\mathbb{G}_m} GU(1)\longrightarrow Gsp_4$. Everything is similar. Moreover, the resulting family of special cycles is the same in fact. These two embedding are conjugated by an element $g \in Gsp_4(\mathbb{Q})$, where $g$ is \[\begin{pmatrix}
 0 & 1 & 0 & 0\\
 1 & 0 & 0 & 0\\
 0 & 0 & 0 & 1\\
 0 & 0 & 1 & 0
 \end{pmatrix}.\] Therefore $Z_K(\textbf{G},\textbf{H})$ is the same as $Z_K(\textbf{G},\overline{\textbf{H}})$. But we can't do such automorphic translation inside $Sh(\textbf{M})$ ($g$ doesn't belong to $\textbf{M}(\mathbb{Q})$). In this viewpoint, $Sh(Gsp_4)$ is a better ambient Shimura variety.

\subsection{Unitary GGP pair and some variants}
\label{unitary}

The usual unitary GGP pair $U(n)\longrightarrow U(n+1)\times U(n)$  has already been studied by Boumasmoud in his thesis (especially tame relations, see \cite{reda2019}). Therefore we won't repeat details and refer to his thesis.

The embedding is standard. Take a totally real field $F$, and an imaginary quadratic extension $F\longrightarrow E$. Let $W$ denote an $n$-dimensional $E$-hermitian space and consider the $n+1$-dimensional $E$-hermitian space $V$=$W\perp E{e_{n+1}}$. Then we have a natural embedding $U(W)\longrightarrow U(V)\times U(W)$ (the second factor map is identity). Taking Weil restriction, we get $\textbf{H}=\mathrm{Res}_{F/\mathbb{Q}} U(W)\longrightarrow \textbf{G}=\mathrm{Res}_{F/\mathbb{Q}}(U(V)\times U(W))$. And it will induce a morphism between Shimura datum $Sh(\textbf{H})\longrightarrow Sh(\textbf{G})$. Then we can apply our general method.

The parametrization (here the involved stabilizer is $\textbf{N}_\textbf{G}(\textbf{H})(Q)$, larger than $\textbf{H}(\mathbb{Q})$) and reciprocity law for geometric connected components can be found in Boumasmoud's thesis \cite{reda2019}. Over a split prime, similarly to our main example, we can reduce to the split reductive case, which corresponds to $GL_n\longrightarrow GL_{n+1}\times GL_n$. We just point out that the split assumption is enough to apply the general argument via root groups. We don't need to analyse $H \backslash G/K$ explicitly as in Boumasmoud's thesis.

For norm relations, the first thing is still to construct a suitable Borel subgroup. We will work over a split prime $v$ of $F$ and  the involved embedding becomes $(H_v,G_v)=(GL_n ,GL_{n+1}\times GL_n)$, after fixing isomorphisms $E\otimes_F F_v \cong F_v \times F_{\overline{v}}, U(W_{F_v})\cong GL(W_{F_v}), U(V_{F_v})\cong GL(V_{F_v})$ etc like in our main example (see section \ref{tamereal} for these isomorphisms). Here $H_v$ (resp. $G_v$) is the $v$-component of $\textbf{H}_{\mathbb{Q}_p}$ ($p$ is the underlying prime) (resp. $\textbf{G}_{\mathbb{Q}_p}$).

Choose an ordered basis $(e_1,...,e_n)$ for $W_{F_v}$  and $(v_1,...,v_n,v_{n+1})$ ($v_i=e_i,1\leq i \leq n+1$) for $V_{F_v}$. Then the embedding $GL_n\longrightarrow GL_{n+1}\times GL_n$ is given by sending the matrix $A$ to $\left(\begin{smallmatrix}
 A  &  \\
    & 1
\end{smallmatrix}\right)$ $\times A$.

The ordered basis $(e_n,...,e_1)$ for $W_{F_v}$ defines a  Borel pair $(T_2,B_2)$ for $GL(n)$ and we use $\overline{B_2}$ to denote the Borel subgroup of $GL(n)$ that is opposed to ${B_2}$ with respect to $T_2$. Consider the ordered basis $(v_1,...v_n,v_1+v_2+...+v_{n+1})$ for $V_{F_v}$, similarly we get a Borel pair $(T_1,B_1)$ for $GL(n+1)$ and the opposite Borel subgroup $\overline{B_1}$.

Consider the Borel pair  $(T,B)$=($T_1 \times T_2$, $B_1 \times B_2$) for $GL(n+1)\times GL(n)$ and the opposite Borel subgroup $\overline{B}=\overline{B_1} \times \overline{B_2}$.   Then the subgroup $H_v$ has trivial intersection with $\overline{B}$. Extending them to integral level, we can verify the spherical condition for $(\mathcal{H}_v,\mathcal{G}_v)$ (integral models of $(H_v,G_v)$) by a dimension argument.

As in the $Gsp_4$ example, we will use explicit matrix computations to calculate conductors.

  Set $v=(1,...,1)^T$($n$ elements), the elements of $H_v(F_v)$ given by the matrix $A$ under the ordered basis $(e_1,..,e_n)$ for $W_{F_v}$  corresponds to the elements of  $G_v(F_v)$ given by  \[\begin{pmatrix}
A & (A-I_n)v \\
0 & 1
\end{pmatrix} \times A\] under the  ordered basis $(v_1,...,v_n,v_1+...+v_{n+1});(e_1,...,e_n)$ for $V_{F_v}\times W_{F_v}$.

 Take a strictly $B_1$-dominant cocharacter $\mu_1=(b_1,...b_{n+1})\in X_*(T_1)$ and a strictly $B_2$-dominant cocharacter $\mu_2=(a_n,...,a_1)\in X_*(T_2)$,  where $b_1>b_2>...>b_{n+1}$ and $a_1<a_2<...a_n$. Then $\mu=(\mu_1, \mu_2) \in X_*(T_1 \times T_2)$ is a strictly $B$-dominant cocharacter. Set $\tau=\mu(\pi)$ (again $\pi$ is a uniformizer of $F_v$), obviously the stabilizer conjecture holds. For any non-negative integer $m$, define the subgroup $H_{v,m}=H_v(F_v)\bigcap \tau^{m}\mathcal{G}_v(\bO_F)\tau^{-m}$ similarly.

As in our main example, we have a formal argument for condition $\heartsuit$ (section \ref{normreal}), the trivial intersection property implies that the conductor grows into infinity. Now we explicit the numbers $c(m,i)$ involved in the abstract norm relation (theorem \ref{absnormrel}).

The main task is still computing conductors of $H_{v,m}$. Suppose $m>0$. Set $w=\min\{a_{i+1}-a_i,b_j-b_{j+1}\}$. Then we have $\det(H_{v,m})= 1+\pi^{mw}\bO_{F_v}$, so the conductor is exactly  $mw$ and  the condition $\clubsuit$ is also satisfied. Denote by  $q$  the cardinality of the residue field. Then the explicit formula for $c(m,i)$ is \[c(m,i)=q^{i(\sum_{k=1}^{n+1}(n+2-2k)b_k+\sum_{k=1}^{n}(n+1-2k)a_{n+1-k}-w)}.\] Then we can apply the general method in section \ref{normreal} to construct special cycles with norm relations.

Next we consider the similitude version. For simplicity we work over $\mathbb{Q}$ and suppose the quadratic imaginary field is $E$. The similitude version is \[\textbf{H}=GU(1,n-1)\times_{\mathbb{G},m} GU(0,1) \longrightarrow \textbf{G}=GU(1,n-1) \times GU(1,n).\] This pair is similar to the usual unitary GGP pair. The advantage is that they are the usual PEL type $GU$ Shimura varieties (the  $U(n)$ Shimura variety is only abelian type Shimura variety), have classical moduli explanations.

As in the $Gsp_4$ example (section \ref{Gsp4}), we introduce  $\textbf{T}^1=U(1)=\ker(GU(1)\overset{Norm}{\longrightarrow}G_{m,\mathbb{Q}})$ and consider the map $q:\textbf{H}\longrightarrow \textbf{T}^1$ sending $(A,z)$ to $\frac{det(A)}{z^{n}}$.

The argument for tame relations is the same as in our previous examples. For norm relations, like in the usual unitary case, we also work over a split prime $p$. This pair will split as $(GL(n)\times GL(1))\times \mathbb{G}_m\longrightarrow GL(n)\times \mathbb{G}_m \times GL(n+1) \times \mathbb{G}_m$, where the three copies of $\mathbb{G}_m$ correspond to the similitude factor. Such similitude factor will influence nothing and we can ignore it. Then we will only need to consider the embedding $GL(n)\times GL(1)\longrightarrow GL(n)\times GL(n+1)$. Then it is very similar to the above usual unitary GGP example. The construction of maximal torus, Borel subgroups etc are the same. Although in this case, we don't have trivial intersection property, the intersection is $GL(1)$ diagonally embedded  in the center of $GL(n)\times GL(1)$, but this intersection lies in the kernel of our map $q$, thus just as in our $Gsp_4$ example, the conductor also grows to infinity (condition $\heartsuit$ in section \ref{normreal} holds). Moreover, the computation of conductors are also the same. Then condition $\clubsuit$ in section \ref{normreal} also holds. Thus we can apply the same argument to produce special cycles with norm relations.

There is another possible variant. We can also consider the embedding \[\textbf{H}=GU(1,n-1)\times_{\mathbb{G}_m} GU(0,1)^{n}\longrightarrow \textbf{G}=GU(1,2n-1).\] Here the product are all over $\mathbb{G}_m$ through the similitude map. And we will use the map $q:\textbf{H}\longrightarrow U(1)$, sending $(A,z_1,...,z_n)$ to $\frac{det(A)}{\prod_{i=1}^{n}z_i}$ to cut out a  one dimensional torus quotient. Similar argument will produce special cycles with tame relations and norm relations over split primes. Here for norm relations, there is a slight change. To obtain the spherical condition, we can not use Borel subgroups, so we will instead work with a suitable parabolic subgroup of $GL(2n)$ corresponding to the partition $(n,n)$ of $2n$.  Direct matrix computations show that the conditions $\clubsuit$ and $\heartsuit$  of section \ref{normreal} hold.  Then the general method still works. And finally we  make a remark. This pair is very similar to the pair $(GU(1,n-1)\times_{\mathbb{G}_m}GU(0,n),GU(1,2n-1))$, which is already studied by Andrew Graham and Syed Waqar Ali Shah in \cite{graham2020}. They also construct cycles with tame relations and norm relations over split primes using the  methods of Loeffler's school.

\subsection{Further development}

\subsubsection{Cohomology theory}
\label{coh}

In this section, we will relate cycles to cohomology theory. The general method is to apply $l$-adic Abel-Jacobi map. We will use the usual spectral sequence to construct it. See \cite{ajmap} for more details.

Let $X$ denote a smooth variety which is equidimensional of dimension $d$ over a field $k$ ($char(k)\neq l$). Let $\overline{X}\times_{k}k^{sep}$ where $k^{sep}$ is a separable closure of $k$. For each $k\subset k^{'} \subset k^{sep}$, we have the Hochschild-Serre spectral sequence \[E_2 ^{i,j}=H^{i}(Gal_{k^{'}},H^{j}(\overline{X},\mathbb{Z}_l(n)))\Rightarrow H^{i+j}(X\times_{k}k^{'}, \mathbb{Z}_l(n) ),\] which degenerates at $E_2$ and gives a map \[\ker(H^{2i}(X\times_{k}k^{'},\mathbb{Z}_l(n))\longrightarrow H^{2i}(\overline{X},\mathbb{Z}_l(n)))\longrightarrow H^{1}(Gal_{k^{'}},H^{2i-1}(\overline{X},\mathbb{Z}_l(n))) .\] Thus we get the $l$-adic Abel-Jacobi map \[AJ_l:(\mathcal{Z}^{n}(X\times_{k}k^{'}))_0 \longrightarrow H^1(Gal_{k^{'}}, H^{2n-1}(\overline{X},\mathbb{Z}_l(n))),\] here $\mathcal{Z}^{i}(X\times_{k}k^{'})$ is the Chow group of codimension $n$ cycles, $(\mathcal{Z}^{n}(X\times_{k}k^{'}))_0$  is the kernel of the cycle class map $\mathcal{Z}^{n}(X\times_{k}k^{'}) \longrightarrow H^{2n}(\overline{X},\mathbb{Z}_l(n))$. We call elements in $(\mathcal{Z}^{n}(X\times_{k}k^{'}))_0$ cohomologically trivial cycles.

 Applying this method to the Shimura variety $X=Sh_\textbf{G}$, which is $2n-1$-dimensional (suppose the quadratic space is $2n+1$-dimensional), we can map cohomologically trivial special cycles (constructed by $Sh_\textbf{H}$) to the cohomology of  $\overline{X}$. Here first we need some techniques to make our special cycles cohomologically trivial.  One possible method  is to apply a parity projector whose  existence is the subject of  the standard sign conjecture. We refer to \cite{morel2019} for more details. See also \cite{rapoport2020} (section 6.2) for unitary Shimura varieties. Another standard method is to use  localization at a suitable  maximal ideal of the Hecke algebra. This technique is widely used when vanishing results away from the middle degree are available. For example, see \cite{lt2020} section 4 for Hilbert modular varieties. Here we just remark that this technique is perfectly suitable for our $Gsp_4$ example. Because such vanishing results for Siegel modular varieties has already been proved in \cite{mt2000} (under some assumptions that have subsequently been established in  \cite{shin}).

 Suppose these special cycles have been trivialized, then we can use them to  construct  an Euler system via the methods in section \ref{tamereal} and section \ref{normreal}. For  norm relations, we need the ordinary condition to construct a norm compatible family. Our tame relations involve the Hecke polynomial. Under the  congruence conjecture of Blasius-Rogawski (see \cite{blasius1994}), this polynomial is related with the characteristic polynomial of the Frobenius, which corresponds to the requirement of tame relation in Euler systems. This conjecture has now been established in many cases. For example, Si Ying Lee  proved this conjecture for Hodge type Shimura varieties under some assumptions in \cite{lee2020}. More recently,  Zhiyou Wu proved this conjecture for all Hodge type Shimura varieties in \cite{wu2021s} as a corollary of the $S=T$ conjecture (see \cite{xiao2017cycles}). This covers the $GSp_4$ example, and the $GU$ analogues of the GGP-pair examples, but it seems that for our main example, where $Sh(\textbf{G})$ is an abelian type Shimura variety, the congruence conjecture remains unknown.

\subsubsection{Arithmetic application}

In the  previous section, we have seen how to relate special cycles to cohomology classes. Thus we can define an Euler system coming from special cycles for the resulting Galois representation. Standard argument of Euler systems will give upper bound for  the Selmer group. For example, Cornut has deduced a rank one result in \cite{cornut2018} under some  assumptions.

Of course for any application, we first have to show that the resulting Euler system is nontrivial.  Such thing is very difficult to verify. The main idea is to  connect special elements to $L$-function. One method is Gross-Zagier formula. For example in Heegner points case, we have Gross-Zagier formula relating heights of Heegner points with derivatives of $L$-functions. Regarding  examples in this paper, unfortunately the Gross-Zagier formula for them is still widely open. The GGP program can be seen as an attempt to generalize the Gross-Zagier formula. For unitary GGP pair, there are many developments now, such as the arithmetic fundamental lemma program, the relative trace formula method etc. We refer to \cite{zhang2018} for more details.  Another strategy to study special cycles is to apply $p$-adic methods. For example, Henri Darmon and Victor Rotger  developed a $p$-adic Gross-Zagier formula to study the generalised Gross-Kudla-Schoen diagonal cycles (see \cite{darmon2014},\cite{darmon2017}). What's more, we can also investigate explicit reciprocity laws. For example, Loeffler and Zerbes established an explicit reciprocity law for their Euler system of $Gsp_4$ (not constructed by cycle class from special cycles), and they proved new cases of the Bloch-Kato conjecture (see \cite{loeffler2020bloch}). They used  higher Hida theory to study $p$-adic $L$-functions (see \cite{loeffler2019higher}).  A different routine is through arithmetic level raising. For example we refer to \cite{wang2020}  for more details about such method.

\section{Appendix: The stabilizer conjecture}
\label{conjecture}

In this section we will work in a local setting. First let's recall our setting in section \ref{absnorm}.

Let $F$ be a $p$-adic field ($p$ is an odd prime) with a uniformizer $\pi$ and denote its residue field $\bO_F/\pi$ by $k$. Let $\textbf{G}$ denote  a reductive group scheme over $\spec(\bO_F)$ and let \textbf{H} be  a closed reductive subgroup scheme of \textbf{G}.  We assume this pair $(\textbf{G},\textbf{H})$ to be spherical over $\spec(\bO_F)$: there exists a Borel subgroup scheme $\overline{\textbf{B}}$ of $\textbf{G}$  such that the $\textbf{H}$-orbit of $[1]$ in $\textbf{G}/\overline{\textbf{B}}$ is open, equivalently, $Lie(\textbf{H})+Lie(\overline{\textbf{B}})=Lie(\textbf{G})$. Take a maximal torus $\textbf{T}$ inside $\overline{\textbf{B}}$ and denote by $\textbf{B}$ the Borel subgroup of $\textbf{G}$ opposed to  $\overline{\textbf{B}}$ with respect to $\textbf{T}$.   Denote $G=\textbf{G}(F)$, $H=\textbf{H}(F)$, $K=\textbf{G}(\bO_F)$ and $K_H=\textbf{H}(\bO_F)=H \bigcap K$. Choose a strict $\textbf{B}$-dominant cocharacter $\mu$ of $\textbf{T}$ and set $\tau=\mu(\pi)$.  For any non-negative integer $m$, define  $\overline{H_m}=\tau^{m}K \tau^{-m} \bigcap \textbf{H}(F)  $, $H_m=\overline{H_m} \bigcap K$. Denote their images in $H^{ab}=\frac{H}{H^{der}}$ by $H_{m}^{ab}$ and $\overline{H_{m}} ^{ab}$. And define $x_m=[\tau^{m}] \in G/K $. We consider the following two properties:

(1) For any $m$, we have $H_m=\overline{H_m}$.

(2) For any $m$, we have $H_m ^{ab}=\overline{H_m }^{ab}$.

Obviously we know $H_m \subset \overline{H_m}$, $H_m ^{ab} \subset \overline{H_m }^{ab}$ and property (1) implies property (2). Property (1) is equivalent to say that the $H$-stabilizer of $x_m \in G/K$  also stabilizes $x_0 \in G/K$.  We propose the following conjectures:

\textbf{Stabilizer Conjecture:} Property (1) always holds.

\textbf{Weak Stabilizer Conjecture:} Property (2) always holds.

In section \ref{normreal}, we need to compute conductors and constants $c(m,i)$.  The weak version is sufficient for our application to norm relations.  At present we don't know how to deal with this weaker conjecture in general directly. We will propose some methods for the stabilizer conjecture. The first observation is that we may assume that $\textbf{G}$ is split. Because we can take a finite unramified extension $F\longrightarrow K$ to split $\textbf{G}$, and consider the spherical pair $(\textbf{H}_{\bO_K},\textbf{G}_{\bO_K})$. If $(\textbf{H}_{\bO_K},\textbf{G}_{\bO_K})$ satisfies the stabilizer conjecture, then $(\textbf{H},\textbf{G})$ also satisfies the stabilizer conjecture.

In fact, our main example has a stronger property than property (1). To state it, we first need some preparation about filtrations and Bruhat-Tits building theory. We refer to \cite{dat} for filtrations, \cite{tits} for buildings, \cite{cor} for some generalization and relations between buildings and filtrations.

Let $\textbf{B}(\textbf{G}_F)$ denote the extended Bruhat-Tits building of $\textbf{G}_F$. Assume $\textbf{G}_F$ is semisimple, then there is a  $\textbf{G}_F$ invariant metric (unique up to scalar) on $\textbf{B}(\textbf{G}_F)$, \[d: \textbf{B}(\textbf{G}_F) \times \textbf{B}(\textbf{G}_F)\longrightarrow  \mathbb{R}_{\geq 0} .\] And $(\textbf{B}(\textbf{G}_F),d)$ is a CAT(0)-space. For general reductive groups, there also exists such invariant metric but it may not be  unique up to scalar. See \cite{cor} section 6.2 for more details. Now we fix  such a metric. In our situation, we can identify $G/K$ with a $G$-orbit of hyperspecial vertices  in $\textbf{B}(\textbf{G}_F)$ and thus view $x_m$ as an hyperspecial vertex in $\textbf{B}(\textbf{G}_F)$.  Through $\textbf{H}(\bO_F)\subset \textbf{G}(\bO_F)$, we can embed the building $\textbf{B}(\textbf{H}_{F})$ for $\textbf{H}_F$ into a closed convex subset of $\textbf{B}(\textbf{G}_F)$ and $x_0=[1]$ becomes a common origin. See \cite{landvogt2000} for more details.  Because $\textbf{B}(\textbf{H}_{F})$ is a closed convex set, there is a convex projection \[pr:\textbf{B}(\textbf{G}_F) \longrightarrow \textbf{B}(\textbf{H}_{F}),\] sending $x$ to its closest point in  $\textbf{B}(\textbf{H}_{F})$. Now we consider the following property:

(0) We have $pr(x_m)=x_0$.

Because the projection $pr$ is $H$-invariant, property (0) implies property (1). We will verify property (0) for our main example. Then it satisfies the stabilizer conjecture. We will use  tools about buildings and filtrations from \cite{cor}. First we will translate property (0) into a property about filtrations.

Following \cite{cor}, let $\Gamma=(\Gamma,+, \leq)$ denote a totally ordered commutative group,  $\mathbb{D}_{\widetilde{S}}(\Gamma)$ denote the diagonalized multiplicative group over a base scheme $\widetilde{S}$ with character group $\Gamma$.  For a reductive group $\widetilde{G}$ over $\widetilde{S}$, Cornut defined and studied the following sequence between $\widetilde{S}$-schemes in \cite{cor} section 2 \[\mathbb{G}^{\Gamma}(\widetilde{G})\overset{Fil}{\longrightarrow}\mathbb{F}^{\Gamma}(\widetilde{G})\overset{t}{\longrightarrow}\mathbb{C}^{\Gamma}(\widetilde{G}),\] here $\mathbb{G}^{\Gamma}(\widetilde{G})=\underline{\mathrm{Hom}}(\mathbb{D}_S(\Gamma),\widetilde{G})$. Also see \cite{cor} section 1 for some motivations. Here we will only work over $S=\spec(R)$ with   $R$ equal to one of the  local rings $F$, $\bO_F$ or $k$ and $\Gamma=(\mathbb{R},+, \leq)$. For simplicity we denote $\textbf{F}(\widetilde{G})=\mathbb{F}^{\mathbb{R}}(\widetilde{G})(R)$ (filtrations). According to \cite{cor} section 4.1.15, there is an additive structure on $\textbf{F}(\widetilde{G})$, \[\textbf{F}(\widetilde{G}) \times \textbf{F}(\widetilde{G}) \overset{+}{\longrightarrow}\textbf{F}(\widetilde{G}).\] The choice of a faithful finite dimensional representation $\rho$  induces a $\widetilde{G}(R)$-invariant scalar product \[\langle-,-\rangle_{\rho}:\textbf{F}(\widetilde{G}) \times \textbf{F}(\widetilde{G}) \longrightarrow \mathbb{R}.\] See \cite{cor} section 4.2 (especially 4.2.2 and 4.2.10) for more details. Here in application we will always fix  such a representation first and simplify $\langle-,-\rangle_{\rho}$ as $\langle-,-\rangle$. Moreover, according to corollary 87 in \cite{cor}, any embedding between groups $\widetilde{H}\hookrightarrow \widetilde{G}$ will induce an embedding of filtrations $\textbf{F}(\widetilde{H}) \hookrightarrow\textbf{F}(\widetilde{G})$.

Now we connect filtrations  and buildings.  There is an action of $\textbf{F}(\textbf{G}_{F})$ on $\textbf{B}(\textbf{G}_{F})$ \[\textbf{B}(\textbf{G}_F) \times \textbf{F}(\textbf{G}_F) \overset{+}{\longrightarrow} \textbf{B}(\textbf{G}_F).\] And $\textbf{B}(\textbf{G}_F)$ becomes an affine $\textbf{F}(\textbf{G}_F)$-space. See \cite{cor} section 6.2 for more details. Be careful about the following sign issue:  for  $x_0$ and $\tau=\mu(\pi)$ as above, we have $\tau(x_0)=x_0+\mathcal{F}_{\mu^{-1}}$, where   $\mathcal{F}_{\mu^{-1}}$ is the filtration determined by $\mu^{-1}$ ($\mathcal{F}_{\mu^{-1}}=Fil(\mu^{-1})$). This can be traced back to  the definition of buildings,  see \cite{tits}, section 1. By \cite{cor} section 6.4.8, there is a reduction map \[\textbf{F}(\textbf{G}_F)\overset{\cong}{\longleftarrow}\textbf{F}(\textbf{G})\overset{red}{\longrightarrow}\textbf{F}(\textbf{G}_k).\] Through this reduction map, according to \cite{cor} section 6.4.13, section 5.5.2 and section 5.5.12, the property (0) is equivalent to the following property:

(A) Over the residue field $k$, for any filtration $\mathcal{F} \in \textbf{F}(\textbf{H}_k)\subset \textbf{F}(\textbf{G}_k)$, we have $ \langle \mathcal{F},red(\mathcal{F}_{\mu^{-1}})\rangle \leq 0$.

We will prove that our main example satisfies property (A) in the following. Before that, we first briefly take  $GL(n)$ as an example for the above concepts. This example is already widely explained, such as in \cite{cor} section 6.1. Here we illustrate it for the benefit of readers.

Consider a $m$-dimensional space $V$ over $F$ and the general linear group $\widetilde{G}=GL(V)$. We will always fix its standard representation to compute scalar products of filtrations.  Its building can be identified with the space of $F$-norms on $V$. And the subset of hyperspecial vertices $hyp(\widetilde{G})$ can be identified with the set of $\bO_F$-lattices in $V$. A filtration on $V$ is a decreasing map with special properties \[\mathcal{F}: \mathbb{R} \longrightarrow \{F\mathrm{-subspaces}\,of\,X\},\quad x \longrightarrow \mathcal{F}^{x}.\] Here the set of subspace of $V$ is partially ordered by inclusion and we require $\mathcal{F}$ to satisfy the following two conditions:

(i) There exists $x \in \mathbb{R}$ such that $\mathcal{F}^{x}=(0)$ and $\mathcal{F}^{-x}=V$, i.e. the filtration is exhaustive and separating.

(ii) Let $\displaystyle \mathcal{F}^{x^{-}}=\bigcap_{y<x}\mathcal{F}^{y}$, then $\mathcal{F}^{x}=\mathcal{F}^{x^{-}}$ (left continuous).

Denote $\displaystyle \mathcal{F}^{>x}=\mathcal{F}^{x^{+}}=\bigcup_{y>x}\mathcal{F}_y$ , $\displaystyle \mathcal{F}^{\geq x}=\bigcup_{y \geq x}\mathcal{F}_y=\mathcal{F}^{x}$ and $gr^{x}_{\mathcal{F}}(V)=\mathcal{F}^{\geq x}/\mathcal{F}^{>x}$. And we refer to  elements in  $\{x|gr^{x}_{\mathcal{F}}(V)\neq 0 \}$ as the breaks  (or jumps) of $\mathcal{F}$. Suppose these elements are $y_l>y_{l-1}>...>y_1$, we may also denote $\mathcal{F}$ as a tuple $(y_l,F^{y_l},y_{l-1},F^{y_{l-1}},...,y_1)$. If all these breaks are integers, then we call $\mathcal{F}$ a $\mathbb{Z}$-filtration, and the subset $\textbf{F}^{\mathbb{Z}}(\widetilde{G})=\mathbb{F}^{\mathbb{Z}}(\widetilde{G})(F)$ consists of $\mathbb{Z}$-filtrations. For a subspace $\widetilde{V}\subset V$, we can restrict $\mathcal{F}$ to $\widetilde{V}$ and denote it by $\mathcal{F}|_{\widetilde{V}}$, $\mathcal{F}|_{\widetilde{V}} ^{x}=\mathcal{F}^{x} \bigcap \widetilde{V}$. For a quotient $\Delta:V  \twoheadrightarrow \overline{V}$ we can also define $\mathcal{F}|_{\overline{V}}$, $x\longrightarrow \Delta(\mathcal{F}^{x})$. Combining these, we can define a filtration $\mathcal{F}|_{V^{'}}$ for any sub-quotient of $V$. Now we describe $\mathbb{G}^{\mathbb{Z}}(\widetilde{G})(F)\overset{Fil}{\longrightarrow}\mathbb{F}^{\mathbb{Z}}(\widetilde{G})(F)$. The elements in the left side correspond to  cocharacters $\alpha:\mathbb{G}_{m} \longrightarrow GL(V)$. Through such $\alpha$, $V$ has a weight decomposition, $V=\oplus_{i}V(i)$, where $V(i)=\{v \in V, \alpha(t)(v)=t^{i}(v)\}$. For any $x \in \mathbb{R}$, define $\displaystyle V^{x}=\bigoplus_{y \geq x, y \in \mathbb{Z}}V(y)  $, then the map $\mathcal{F}:x \longrightarrow \mathcal{F}^{x}=V^{x}$ is a $\mathbb{Z}$-filtration and it is exactly $Fil(\alpha)$. If we take a Borel pair $(\widetilde{B},\widetilde{T})$, then the map $Fil:X^{+}(\widetilde{T})\longrightarrow \textbf{F}(\widetilde{G})$ is compatible with the addition maps. Now we describe the scalar product. For a filtration $\mathcal{F}$, we  define $\deg(\mathcal{F}|V)=\sum_{x}x \dim(gr^{x}_{\mathcal{F}}(V))$.  For two filtrations $\mathcal{F}_1$, $\mathcal{F}_2$, their scalar product is as follow \[\langle \mathcal{F}_1,\mathcal{F}_2\rangle=\sum_{x,y}xy \dim \frac{\mathcal{F}_1 ^{x}\bigcap \mathcal{F}_2^{y}}{\mathcal{F}_1^{>x}\bigcap \mathcal{F}_2^{y}+\mathcal{F}_1^{x}\bigcap \mathcal{F}_2^{>y}}\] \[=\sum_x x \deg(\mathcal{F}_2|_{gr_{\mathcal{F}_1}^{x}(V)})=\sum_x x \deg(\mathcal{F}_1|_{gr_{\mathcal{F}_2}^{x}(V)}).\]Notice that we can rewrite $\deg(\mathcal{F}|V)=\langle \mathcal{F}, (0)\rangle$.

Now we illustrate the relation between filtrations and buildings. The action of $\textbf{F}^\mathbb{Z}(\widetilde{G})=\mathbb{F}^{\mathbb{Z}}(\widetilde{G})(F)$ on $Hyp(\widetilde{G})$ is as follow \[L+\mathcal{F}=\sum_{i\in\mathbb{Z}}\frac{1}{\pi^{i}}L \bigcap \mathcal{F}^{i}.\] Take a lattice $L_0$, it will give a reductive integral model $\mathcal{G}=GL(L_0)$ over $\spec(\bO_F)$. The reduction map for filtrations is as follows \[red(\mathcal{F}):x\longrightarrow (\mathcal{F}^x \bigcap L_0)/\pi.\]

Now let's come back to the proof of (A) for our main example. First let's recall its construction in section \ref{normreal}. Let $F \longrightarrow E$ denote a quadratic unramified extension ($E$ is a field or split as $F \times F$) and let $\bO_E$ denote the integral closure of $\bO_F$ in $E$.  Let  $(V,\phi)$ denote a $2n+1$-dimensional quadratic $F$-space with an $n$-dimensional hermitian $E$-space $(W,\psi)$. Our main example (over generic fiber) is $(SO(V),U(W))$. Now recall the notation of  special basis (see section \ref{normreal}). Take an element $\eta \in \bO_{E}^*$ with $\eta+\overline{\eta}=0$, we have the following orthogonal $F$-basis for $V$, \[\beta= (v_0,w_1,v_1,w_2,...,w_n,v_n),\] where $\{w_1,...,w_n\}$ is an orthogonal $E$-basis for $W$, $V=W\perp Fv_n$,  $v_{i-1}=\eta w_i$ $(1 \leq i \leq n)$ and $\phi(w_i,w_i)+\phi(v_i,v_i)=0$. It defines two orthogonal decomposition of $V$, \[V=E{w_1}\perp...\perp Ew_n\perp F v_n=Fv_0\perp H_1\perp...\perp H_n,\] where $H_i=Fw_i\perp Fv_i$ is an hyperbolic $F$-plane with isotropic lines spanned by $e_{\pm i}=\frac{1}{2}(v_i \pm w_i)$. Define $e_0=v_0$ and consider the ordered basis $(e_n,..,e_{-n})$.  Consider the $\bO_F$-lattice $L_0$ spanned by this basis, it is a self-dual lattice and gives us a reductive integral model $\textbf{G}=SO(L_0)$ over $\spec(\bO_F)$. Through $SO(L_0)\hookrightarrow GL(L_0)$ we get  a Borel pair $(\textbf{T},\textbf{B})$ for $\textbf{G}$ ($\textbf{B}$ corresponds to upper triangular matrices under this ordered basis). Consider the $\bO_E$-lattice $\widetilde{L}$ spanned by $(w_1,...,w_n)$, it defines a reductive integral model $\textbf{H}=U(\widetilde{L})$. It is easy to check $\textbf{H} \bigcap \overline{\textbf{B}}=1$ ($\overline{\textbf{B}}$ is the Borel group opposite to $\textbf{B}$ with respect to $\textbf{T}$) and then $(\textbf{H},\textbf{G},\overline{\textbf{B}})$ satisfies the spherical condition by dimension reason.

Fix the standard representation for orthogonal group to compute scalar product of filtrations. Take a strict $\textbf{B}$-dominant cocharacter $\mu \in X_* ^{+}(\textbf{T})$ associated to $(s_n,...,s_1,0,-s_1,...,-s_n)$, where $s_i$ are integers with $0<s_1<...<s_n$, we will verify property (A) for $red(\mathcal{F}_{\mu^{-1}})$. We first make some simplifications to our notations. Let $V_k$ denote the $2n+1$-dimensional quadratic $k$-space $L_0/\pi$, $h$ denote $\bO_E/\pi$, $\theta$ denote $\eta/\pi$, $W_k$ denote the $n$-dimensional hermitian $h$-space $\widetilde{L}/\pi$, $G_k$ denote the orthogonal group $\textbf{G}_k=SO(V_k)$ and $H_k$ denote the unitary group $\textbf{H}_k=U(W_k)$. Define the ordered basis $(f_n,...,f_{-n})=(\overline{e_{-n}},...,\overline{e_n})$ for $V_k$ and define a cocharacter $\zeta=(s_n,...,s_1,0,-s_1,...,-s_n)$ under this ordered basis. Then $red(\mathcal{F}_{\mu^{-1}})=\mathcal{F}_{\zeta}$.

Through the natural embedding $SO(V_k)\hookrightarrow GL(V_k)$, we embed $\textbf{F}(G_k)$ into $\textbf{F}(GL(V_k))$.

Denote $s_0=0$, $s_{-i}=-s_i$ ($1\leq i \leq n$). For each $-n\leq i \leq n$, define $\displaystyle V_i=\oplus_{j \geq i}k(f_j)$. The filtration $\mathcal{F}_{\zeta}$ corresponds to this tuple $(s_n,V_n,...,s_{-n})$.  It has a self-dual property, i.e. $V_{-i}=V_{i+1}^{\perp}$ ($0 \leq i \leq n-1$) and $s_{-i}=-s_i$ ($0\leq i \leq n$). In fact, any filtration in $\textbf{F}(G_k)$ has this property.

The following observation shows that the filtration $\mathcal{F}_{\zeta}$ is "orthogonal" to the $h$-linear structure.

\begin{lem}
For any vector $v=\sum_{i=1}^{m}c_if_i$ with $1 \leq m \leq n-1$ and $c_{m}\neq 0$, we have $\theta(v)\in V_{-m-1}-V_{-m}$.
\end{lem}
\begin{proof}
   For any $0 \leq m \leq n-1$, we have $V_{-m}=\oplus_{j={m+1}}^{n}k(f_j)\oplus k(\theta(\overline{w_{m+1}})) \oplus_{j=1}^{m} h(\overline{w_j})$: The right side is a subspace of the left side and they have the same dimension ($n+m+1$ dimensional), thus they are equal.

   Now for such vector $v$, we have $\theta(v)=\sum_{i=1}^{m-1}c_i\theta (f_i)+c_{m}(\theta^{2}(\overline{w_{m+1}})-\theta(\overline{w_{m}}))$. Then it lies in $V_{-m-1}$. Because $c_m \theta^2 $ is nonzero elements in $k$ and $\overline{w_{m+1}} \notin V_{-m}$, thus $\theta(v)\in V_{-m-1}-V_{-m}$.
\end{proof}
Now we first verify property (A) for those minimal type filtrations. For any $h$-isotropic subspace $X  \subset W_k $ and a positive real number $x$, we have a filtration $\mathcal{F}_{X} \in \textbf{F}(H_k)$, $\mathcal{F}_X=(x,X,0,X^{\perp},-x)$. We have the following lemma:

\begin{lem}
\label{deg}
We have $\deg(\mathcal{F}_{\zeta}|X)  \leq 0$.
\end{lem}
\begin{proof}
   Denote the breaks set of $\mathcal{F}_{\zeta}|X$ as $Se$. Then $Se \subset \{s_n,...,s_{-n}\} $. And we divide it into two parts, $Se=Se^{+} \coprod Se^{-}$, where $Se^{+}=\{s_i \in Se, 1\leq i \leq n\}$, $Se^{-}=\{s_i, -n \leq i \leq 0\}$.    Denote the degree $\deg(\mathcal{F}_{\zeta}|X)  $ as $S$, so that $S=\sum_{s_i \in Se}s_i=S^{+}+S^{-}$ with $S^{+}=0+\sum_{s_i \in Se^{+}}s_i$, $S^{-}=0+\sum_{s_i \in Se^{-}}s_i$. Then $S^+ \geq 0$ and $S^{-} \leq 0$.

   If $S^{+}=0$, we're done. If $S^{+}>0$, then $Se^{+} \neq \emptyset$ and we will cancel its contribution through $S^{-}$.

   Now define a function $Ma:V-{0}\longrightarrow \mathbb{Z}$.

    For any nonzero vector $v$, denote  $v=\sum_{l=-n}^{i}c_l f_l$ with $c_i \neq 0$, then define $Ma(v)=i$.

    We next define a  "\textbf{reverse}" function $r: Se^{+}\longrightarrow \mathbb{Z}$.

   Take any $s_i\in Se^{+}$, then $i\leq n-1$, and we  define  $\displaystyle r(s_i)=\min_{v \in (V_{i} \bigcap X-V_{i+1} \bigcap X)} Ma(v)$.

    By definition we know $r(s_i)\geq i$ and it is injective:

    Suppose there exists $i<j$, with $r(s_i)=r(s_j)$. By definition, there exists $v_i \in V_{i} \bigcap X -V_{i+1} \bigcap X$ (resp $v_j \in V_{j} \bigcap X-V_{j+1} \bigcap X$) such that $v_i=\sum_{l=i}^{r(s_i)}c_l f_l$ (resp $v_j=\sum_{l=j}^{r(s_j)}d_l f_l$) with $ c_{r(s_i)} \neq 0 $ (resp $d_{r(s_j)}\neq0$).  Then consider the vector $v=d_{r(s_i)}v_j-c_{r(s_j)}v_i \in V_{i} \bigcap X -V_{i+1} \bigcap X$, but $Ma(v)<Ma(v_j)$, this contradicts  the definition of $r$.

   For each $s_i \in Se^{+}$, there exists $v_i \in V_{i} \bigcap X -V_{i+1} \bigcap X$ such that $v_i=\sum_{l=1}^{r(s_i)}c_l f_l$ with $ c_{r(s_i)} \neq 0 $. By the above lemma \ref{deg}, $\theta(v_i) \in V_{-r(i)-1}-V_{-r(i)}$. Then $s_{-r(i)-1} \in Se^{-}$. Consider the sum $\displaystyle S^{'}=\sum_{s_i \in Se^{+}}s_i-s_{r(s_i)+1}$. Then $S=S^{'}+S^{''}$ with $S^{''}\leq 0$, $S^{'}<0$ (due to $r(s_i)\geq i$). Thus the degree is non-positive.
\end{proof}
By self-dual properties, we have $\langle \mathcal{F}_X, \mathcal{F}_{\zeta} \rangle=2x(\deg(\mathcal{F}_{\zeta}|X)).$ Thus this lemma  verifies property (A) for minimal type filtrations. Moreover, it also helps us to reduce the general case into the minimal case.

\begin{thm}
For any $\mathcal{F}_1 \in \textbf{F}(H_k)$, we have $\langle F_1, F_{\zeta} \rangle \leq 0$.
\end{thm}
\begin{proof}
The case $\mathcal{F}_1=(0)$ is trivial. For other cases, the tuple corresponding to $\mathcal{F}_1$ will be  of the form $(a_m,X_m,...,a_{-m})$, where $a_i+a_{-i}=0$, $X_{-i}=X_{i+1}^{\perp}$ ($0 \leq i \leq m-1$), $X_i$ ($1\leq i \leq m$) are isotropic $h$-spaces. We have shown the case $m=1$. Now we do induction for  $m\geq 2$.

Consider another filtration $\mathcal{F}_2 \in \textbf{F}(H_k)$, defined by the tuple $(a_{m-1},X_{m-1},...,a_{-(m-1)})$. By self-duality properties, we have $\langle \mathcal{F}_1, \mathcal{F}_{\zeta}\rangle=\langle \mathcal{F}_2, \mathcal{F}_{\zeta} \rangle +2(a_{m}-a_{m-1}) \deg (\mathcal{F}_{\zeta}|X_m)$. Because $a_m-a_{m-1}>0$, by  lemma  \ref{deg} and induction argument, we're done.
\end{proof}
We make some remarks.

$\bullet$ We can relax the condition (for cocharacter $\zeta$) $s_1<...<s_m$ into $s_1\leq ...\leq s_m$, the above theorem still holds and its proof is similar.

$\bullet$ Property (A) holds for any $\mu$ which is stricly dominant with respect to a Borel subgroup in the open $\textbf{H}$-orbit. Indeed, property (A) only depens upon the filtration $\mathcal{F}_{\mu^{-1}}$, and any such filtration is split by some special basis $\beta$ as above. In other words,  property (A) is an intrinsical property of  our spherical pair $(\textbf{H},\textbf{G})$.

Now we discuss other cases, especially the eight infinite families of indecomposable pairs $(\textbf{H},\textbf{G})$ in \cite{loeffler2019} section 6.  Using similar methods, we can show that the following four  kinds of families also have property (A):

$\bullet (GL(n),Sp(2n))$ $\bullet (SO(n),GL(n))$

 $ \bullet (SO(n)\times SO(n+1), SO(2n+1))$ $ \bullet (SO(n)\times SO(n), SO(2n))$

 But in general, we shouldn't expect property (A) for spherical pairs. Below is a classical counterexample.

Suppose the pair is $(G_0, G_0 \times G_0)$ over $\spec(\bO_F)$ with $G_0$ being semi-simple.  Take a Borel pair  ($T$, $B_1$) for $G_0$ and let $B_2$ be the Borel subgroup of $G_0$ opposed to $B_1$ with respect to $T$. Then $B_1 \times B_2$ is a Borel subgroup of $G_0 \times G_0$ and $(G_0, G_0 \times G_0)$ is a spherical pair. This is a standard example. Consider a strict $B_1$-dominant cocharacter $\mu_1 \in X_*(T)$ and a strict $B_2$-dominant cocharacter $\mu_2 \in X_*(T)$; then for any positive integer $N$,  we have a strict $B_1\times B_2$-dominant cocharacter $\mu_N=((\mu_1)^N,\mu_2)\in X_*(T\times T)$. The corresponding filtration involved in property (0) is $\mathcal{F}_{\mu_N}=(N(\mathcal{F}_{\mu_1 ^{-1}}), \mathcal{F}_{\mu_2^{-1}})\in \textbf{F}(G_{0,F}\times G_{0,F})=\textbf{F}(G_{0,F})\times \textbf{F}(G_{0,F})$.  Consider the filtration $\mathcal{F}_{N}=N(\mathcal{F}_{\mu_1 ^{-1}}) \in \textbf{F}(G_{0,F}) $, then $\langle \mathcal{F}_{N} ,\mathcal{F}_{\mu_N} \rangle= N^2 \langle \mathcal{F}_{\mu_1 ^{-1}}, \mathcal{F}_{\mu_1 ^{-1}}  \rangle +N \langle \mathcal{F}_{\mu_1 ^{-1}} ,\mathcal{F}_{\mu_2^{-1}} \rangle $; since $\langle \mathcal{F}_{\mu_1 ^{-1}}, \mathcal{F}_{\mu_1 ^{-1}}  \rangle  >0$, this scalar product will be positive for $N \gg 0$. Similar argument over the residue field $k$ will produce a counterexample to property (A).

Moreover, we can produce counterexamples to property (A) for the three remaining families (the idea is similar except for the last family):

$\bullet (GL(n), GL(n)\times GL(n+1))$ $ \bullet (SO(n), SO(n)\times SO(n+1))$

$\bullet (Sp(2n),SL(2n+1))$.

Now we turn to our original conjecture concerning property (1). We have a general strategy to verify it for symmetric pairs.

The starting input is a reductive group scheme $\textbf{G}$ over $\spec(\bO_F)$ with a nontrivial involution $\theta$. Let $\textbf{H}$ denote the $\theta$ fixed subgroup. Such a pair $(\textbf{H},\textbf{G})$ is called a symmetric pair. We assume that  there exists a Borel subgroup $\textbf{B}$ of $\textbf{G}$ such that $\theta(\textbf{B})=\overline{\textbf{B}}$ is opposed to $\textbf{B}$. Then $\textbf{T}=\theta(\textbf{B}) \bigcap \textbf{B}$ is a $\theta$-stable maximal torus and  $(\textbf{G},\textbf{H})$ is a spherical pair (see \cite{helminck1993}).

 Here we make some remarks:

 $\bullet$ A parabolic group $\textbf{P}$ such that $\theta(\textbf{P})$ is opposed  to $\textbf{P}$ is called a $\theta$-split parabolic subgroup. In general, we can replace $\textbf{B}$ by a minimal $\theta$-split parabolic subgroup,  our argument below still works.

 $\bullet $ Over a field $F$, there are many studies about symmetric pairs. For example, there always exists a $\theta$-stable maximal torus and under suitable conditions, there exists a nontrivial minimal $\theta$-split parabolic subgroup, we refer to \cite{helminck1993} for more details. Moreover, the connected component of the fixed subgroup is a reductive group  by \cite{prasad2002}.

Now take a strictly $\textbf{B}$-dominant cocharacter $\mu$ of $\textbf{T}$ and define $\tau=\mu(\pi)$, we will show that property (1) holds.

Because $\mu$ is strictly $\textbf{B}$-dominant, the cocharacters $\eta=\theta(\mu)$ and $\mu^{-1}\eta$ are all  strict $\overline{\textbf{B}}$-dominant. Define $\xi=(\mu^{-1}\eta)(\pi)$ and let  $\textbf{N}$ be  the unipotent radical of $\textbf{B}$ and $\overline{\textbf{N}}$ be the unipotent radical of $\overline{\textbf{B}}$.

For any positive integer $m$, take an element $x \in \overline{H_m}=\textbf{H}(F)\bigcap \tau^{m}\textbf{G}(\bO_F)\tau^{-m}$ and suppose $x=\tau^{m}y\tau^{-m}$ ($y \in \textbf{G}(\bO_F)$).

Then we know $x=\theta(x)$, thus $\tau^{m}y \tau^{-m}=\theta(\tau)^{m}\theta(y)\theta(\tau)^{-m}$.

Therefore $y=\xi^{m}\theta(y)\xi^{-m}$. Since $y$ and $\theta(y)$ belong to $\textbf{G}(\bO_F)$, this implies  that $y$ lies in the big Bruhat cell $\overline{\textbf{N}}\times \textbf{T} \times \textbf{N}$, and moreover $y$ decomposes as $\xi^{m}\overline{n_1} \xi^{-m} \times t \times n_2$ with  $\overline{n_1} \in \overline{\textbf{N}}(\bO_F)$ and $n_2 \in \textbf{N}(\bO_F)$.

Thus $x=\tau^{m}y\tau^{-m}= \eta({\pi})^{m} \overline{n_1} \eta({\pi})^{-m} \times t \times \tau^{m} n_2 \tau^{-m} $. Then property (1) holds since $\eta$ is $\overline{\textbf{B}}$-dominant and $\mu$ is $\textbf{B}$-dominant.

Now we list some examples using this idea. Consider the diagonal embedding mentioned above, $G_0\longrightarrow G_0 \times G_0$. Here we don't need $G_0$ to be semisimple. Consider the involution $\theta$ on $G_0 \times G_0$ defined by swapping factors, $(x,y)\longrightarrow (y,x)$. Then the fixed subgroup is exactly $G_0$ and the Borel subgroup $B \times \overline{B}$ is a  $\theta$-split minimal parabolic subgroup. Thus we can apply the above method to show property (1). Another example is $SO(n)\longrightarrow GL(n)$ and the involution $\theta$ is given by $A\longrightarrow A^{-T}$.

Involutions can also be used to study property (A). If there exists an involution $\theta$ of $\textbf{F}(\textbf{G}_k)$ that preserves the scalar product, fixes $\textbf{F}(\textbf{G}_k)$ and such that $\theta(red(\mathcal{F}_{\mu^{-1}}))+red(\mathcal{F}_{\mu^{-1}})=0$, then property (A) holds:
According to \cite{cor} corollary 92, we have \[2\langle \mathcal{F}_1, red(\mathcal{F}_{\mu^{-1}})\rangle =\langle \mathcal{F}_1, red(\mathcal{F}_{\mu^{-1}}) \rangle + \langle \mathcal{F}_1, \theta(red(\mathcal{F}_{\mu^{-1}}))\rangle \leq \langle  \mathcal{F}_1,red(\mathcal{F}_{\mu^{-1}})+\theta(red(\mathcal{F}_{\mu^{-1}}))\rangle =0.\]

Finally, we mention that in some concrete cases beyond symmetric pairs, we can also use explicit calculations to show property (1). See the example $(GU(1)\times_{\mathbb{G}_m}GL_2,Gsp_4)$ in section \ref{Gsp4} and the unitary GGP pair in section \ref{unitary}.

\newpage

\bibliographystyle{plain}
\bibliography{reference}

\begin{thebibliography}{10}

\bibitem{blasius1994}
Don Blasius and Jonathan~D. Rogawski.
\newblock {Zeta functions of {S}himura varieties}.
\newblock In {\em Motives ({S}eattle, {WA}, 1991)}, volume~55 of {\em Proc.
  Sympos. Pure Math.}, pages 525--571. Amer. Math. Soc., Providence, RI, 1994.

\bibitem{reda2019}
Mohamed~R{\'e}da Boumasmoud.
\newblock {\em {Generalized Norm-compatible Systems on Unitary Shimura
  Varieties}}.
\newblock PhD thesis, Ecole Polytechnique F{\'e}d{\'e}rale de Lausanne, 2019.

\bibitem{cornut2011}
Christophe Cornut.
\newblock On {$p$}-adic norms and quadratic extensions, {II}.
\newblock {\em Manuscripta Math.}, 136(1-2):199--236, 2011.

\bibitem{cornut2018}
Christophe Cornut.
\newblock {An Euler system of Heegner type}.
\newblock {\em preprint}, 22, 2018.

\bibitem{cor}
Christophe Cornut.
\newblock {\em {Filtrations and buildings}}, volume 266.
\newblock American Mathematical Society, 2020.

\bibitem{darmon2014}
Henri Darmon and Victor Rotger.
\newblock {Diagonal cycles and Euler systems I: A p-adic Gross-Zagier formula}.
\newblock {\em Ann. Sci. {\'E}c. Norm. Sup{\'e}r.(4)}, 47(4):779--832, 2014.

\bibitem{darmon2017}
Henri Darmon and Victor Rotger.
\newblock Diagonal cycles and {E}uler systems {II}: {T}he {B}irch and
  {S}winnerton-{D}yer conjecture for {H}asse-{W}eil-{A}rtin {$L$}-functions.
\newblock {\em J. Amer. Math. Soc.}, 30(3):601--672, 2017.

\bibitem{dat}
Jean-Fran\c{c}ois Dat, Sascha Orlik, and Michael Rapoport.
\newblock {\em Period domains over finite and {$p$}-adic fields}, volume 183 of
  {\em Cambridge Tracts in Mathematics}.
\newblock Cambridge University Press, Cambridge, 2010.

\bibitem{Deligne1971}
Pierre Deligne.
\newblock Vari\'{e}t\'{e}s de {S}himura: interpr\'{e}tation modulaire, et
  techniques de construction de mod\`eles canoniques.
\newblock In {\em Automorphic forms, representations and {$L$}-functions
  ({P}roc. {S}ympos. {P}ure {M}ath., {O}regon {S}tate {U}niv., {C}orvallis,
  {O}re., 1977), {P}art 2}, Proc. Sympos. Pure Math., XXXIII, pages 247--289.
  Amer. Math. Soc., Providence, R.I., 1979.

\bibitem{sga3}
Philippe Gille and Patrick Polo, editors.
\newblock {\em Sch\'{e}mas en groupes ({SGA} 3). {T}ome {III}. {S}tructure des
  sch\'{e}mas en groupes r\'{e}ductifs}, volume~8 of {\em Documents
  Math\'{e}matiques (Paris) [Mathematical Documents (Paris)]}.
\newblock Soci\'{e}t\'{e} Math\'{e}matique de France, Paris, 2011.
\newblock S\'{e}minaire de G\'{e}om\'{e}trie Alg\'{e}brique du Bois Marie
  1962--64. [Algebraic Geometry Seminar of Bois Marie 1962--64], A seminar
  directed by M. Demazure and A. Grothendieck with the collaboration of M.
  Artin, J.-E. Bertin, P. Gabriel, M. Raynaud and J-P. Serre, Revised and
  annotated edition of the 1970 French original.

\bibitem{graham2020}
Andrew Graham and Syed Waqar~Ali Shah.
\newblock {Anticyclotomic Euler systems for unitary groups}.
\newblock {\em arXiv preprint arXiv:2001.07825}, 2021.

\bibitem{helminck1993}
Aloysius~G Helminck and Shu~Ping Wang.
\newblock {On rationality properties of involutions of reductive groups}.
\newblock {\em Adv. Math}, 99(1):26--96, 1993.

\bibitem{jetchev2014}
Dimitar~P Jetchev.
\newblock {Hecke and Galois properties of special cycles on unitary Shimura
  varieties}.
\newblock {\em arXiv preprint arXiv:1410.6692}, 2014.

\bibitem{landvogt2000}
Erasmus Landvogt.
\newblock {Some functorial properties of the {B}ruhat-{T}its building}.
\newblock {\em J. Reine Angew. Math.}, 518:213--241, 2000.

\bibitem{lang}
Serge Lang.
\newblock Algebraic groups over finite fields.
\newblock {\em Amer. J. Math.}, 78:555--563, 1956.

\bibitem{lee2020}
Si~Ying Lee.
\newblock {Eichler-{S}himura {R}elations for {S}himura {V}arieties of {H}odge
  {T}ype}.
\newblock {\em arXiv preprint, arXiv:2006.11745}, 2021.

\bibitem{lt2020}
Yifeng Liu and Yichao Tian.
\newblock Supersingular locus of {H}ilbert modular varieties, arithmetic level
  raising and {S}elmer groups.
\newblock {\em Algebra Number Theory}, 14(8):2059--2119, 2020.

\bibitem{loeffler2019}
David Loeffler.
\newblock {Spherical varieties and norm relations in Iwasawa theory}.
\newblock {\em arXiv preprint arXiv:1909.09997}, 2019.

\bibitem{loeffler2019higher}
David Loeffler, Vincent Pilloni, Christopher Skinner, and Sarah~Livia Zerbes.
\newblock {Higher Hida theory and p-adic L-functions for GSp(4)}.
\newblock {\em arXiv preprint arXiv:1905.08779}, 2019.

\bibitem{loeffler2017}
David Loeffler, Chris Skinner, and Sarah~Livia Zerbes.
\newblock {Euler systems for GSp(4)}.
\newblock {\em arXiv preprint arXiv:1706.00201}, 2017.

\bibitem{loeffler2020}
David Loeffler, Christopher Skinner, and Sarah~Livia Zerbes.
\newblock {An Euler system for GU(2, 1)}.
\newblock {\em arXiv preprint arXiv:2010.10946}, 2020.

\bibitem{loeffler2020bloch}
David Loeffler and Sarah~Livia Zerbes.
\newblock {On the Bloch-Kato conjecture for GSp(4)}.
\newblock {\em arXiv preprint arXiv:2003.05960}, 2020.

\bibitem{milne}
J.~S. Milne.
\newblock Introduction to {S}himura varieties.
\newblock In {\em Harmonic analysis, the trace formula, and {S}himura
  varieties}, volume~4 of {\em Clay Math. Proc.}, pages 265--378. Amer. Math.
  Soc., Providence, RI, 2005.

\bibitem{mt2000}
Abdellah Mokrane and Jacques Tilouine.
\newblock Cohomology of {S}iegel varieties with {$p$}-adic integral
  coefficients and applications.
\newblock {\em Ast\'{e}risque}, (280):1--95, 2002.
\newblock Cohomology of Siegel varieties.

\bibitem{morel2019}
Sophie Morel and Junecue Suh.
\newblock {The standard sign conjecture on algebraic cycles: The case of
  Shimura varieties}.
\newblock {\em Journal f{\"u}r die reine und angewandte Mathematik (Crelles
  Journal)}, 2019(748):139--151, 2019.

\bibitem{ajmap}
Jan Nekov\'{a}\v{r}.
\newblock {$p$}-adic {A}bel-{J}acobi maps and {$p$}-adic heights.
\newblock In {\em The arithmetic and geometry of algebraic cycles ({B}anff,
  {AB}, 1998)}, volume~24 of {\em CRM Proc. Lecture Notes}, pages 367--379.
  Amer. Math. Soc., Providence, RI, 2000.

\bibitem{nek}
Jan Nekov\'{a}\v{r}.
\newblock On the parity of ranks of {S}elmer groups. {II}.
\newblock {\em C. R. Acad. Sci. Paris S\'{e}r. I Math.}, 332(2):99--104, 2001.

\bibitem{prasad2002}
Gopal Prasad and Jiu-Kang Yu.
\newblock {On finite group actions on reductive groups and buildings}.
\newblock {\em Inventiones Mathematicae}, 147(3):545--560, 2002.

\bibitem{rapoport2020}
Michael Rapoport, Brian Smithling, and Wei Zhang.
\newblock {Arithmetic diagonal cycles on unitary Shimura varieties}.
\newblock {\em Compositio Mathematica}, 156(9):1745--1824, 2020.

\bibitem{shin}
Sug~Woo Shin and Nicolas Templier.
\newblock On fields of rationality for automorphic representations.
\newblock {\em Compos. Math.}, 150(12):2003--2053, 2014.

\bibitem{tits}
J.~Tits.
\newblock Reductive groups over local fields.
\newblock In {\em Automorphic forms, representations and {$L$}-functions
  ({P}roc. {S}ympos. {P}ure {M}ath., {O}regon {S}tate {U}niv., {C}orvallis,
  {O}re., 1977), {P}art 1}, Proc. Sympos. Pure Math., XXXIII, pages 29--69.
  Amer. Math. Soc., Providence, R.I., 1979.

\bibitem{wang2020}
Haining Wang.
\newblock {Arithmetic level raising on triple product of Shimura curves and
  Gross-Schoen Diagonal cycles I: Ramified case}.
\newblock {\em arXiv preprint arXiv:2004.00555}, 2020.

\bibitem{wed}
Torsten Wedhorn.
\newblock Congruence relations on some {S}himura varieties.
\newblock {\em J. Reine Angew. Math.}, 524:43--71, 2000.

\bibitem{wu2021s}
Zhiyou Wu.
\newblock {S=T for Shimura Varieties and p-adic Shtukas}.
\newblock {\em arXiv preprint arXiv:2110.10350}, 2021.

\bibitem{xiao2017cycles}
Liang Xiao and Xinwen Zhu.
\newblock {Cycles on Shimura varieties via geometric Satake}.
\newblock {\em arXiv preprint arXiv:1707.05700}, 2017.

\bibitem{zhang2018}
Wei Zhang.
\newblock {Periods, cycles, and L-functions: a relative trace formula
  approach}.
\newblock In {\em Proceedings of the International Congress of Mathematicians:
  Rio de Janeiro 2018}, pages 487--521. World Scientific, 2018.

\end{thebibliography}

Ruishen Zhao, IMJ-PRG, Sorbonne University,  4 place Jussieu, 75005 Paris, France

E-mail:ruishen.zhao@imj-prg.fr


\end{spacing}
\end{document}